\documentclass{article}
\usepackage{amsmath, amsthm, amssymb, amscd}
\usepackage{enumerate}
\usepackage{verbatim}
\usepackage{tikz}
\usepackage{tikz-cd}
\usepackage[noadjust]{cite}
\usepackage{hyperref}
\usetikzlibrary{matrix,arrows,decorations.pathmorphing}

\DeclareMathOperator{\diam}{diam}
\DeclareMathOperator{\dist}{dist}
\DeclareMathOperator{\interior}{int}

\title{Bi-Lipschitz Pieces between Manifolds}
\author{Guy C. David}
\date{December 13, 2013}
\begin{document}
\maketitle
\newtheorem{thm}{Theorem}[section]
\newtheorem{lemma}[thm]{Lemma}
\newtheorem{prop}[thm]{Proposition}
\newtheorem{cor}[thm]{Corollary}
\newtheorem{claim}[thm]{Claim}

\theoremstyle{remark}
\newtheorem{rmk}[thm]{Remark}

\theoremstyle{definition}
\newtheorem{definition}[thm]{Definition}

\numberwithin{equation}{section}

\begin{abstract}
A well-known class of questions asks the following: If $X$ and $Y$ are metric measure spaces and $f:X\rightarrow Y$ is a Lipschitz mapping whose image has positive measure, then must $f$ have large pieces on which it is bi-Lipschitz? Building on methods of David (who is not the present author) and Semmes, we answer this question in the affirmative for Lipschitz mappings between certain types of Ahlfors $s$-regular, topological $d$-manifolds. In general, these manifolds need not be bi-Lipschitz embeddable in any Euclidean space. To prove the result, we use some facts on the Gromov-Hausdorff convergence of manifolds and a topological theorem of Bonk and Kleiner. This also yields a new proof of the uniform rectifiability of some metric manifolds.
\end{abstract}

\tableofcontents

\section{Introduction}

There are nowadays many different theorems of the following general form: Let $(X, d, \mu)$ and $(Y,\rho, \nu)$ be metric measure spaces (satisfying some assumptions), and let $f:X\rightarrow Y$ be a Lipschitz map whose image has positive $\nu$-measure. Then $f$ must be bi-Lipschitz on a subset of large measure, in a quantitative way.

This class of theorems is not true in general, and later on in Section \ref{counterex} we will mention some interesting cases where it fails. However, there are a number of situations in which results of this form can be proven.

The idea started with \cite{Da88}, where David\footnote{The Guy David mentioned here and in the references is a professor at Universit\'e Paris-Sud and has no relation to the author of this paper, who is a graduate student at UCLA. The author wishes to apologize for any confusion generated by this amusing coincidence.} examined the case in which $(X,d,\mu)$ is Ahlfors $n$-regular and $Y$ is $\mathbb{R}^n$ with the standard metric and Lebesgue measure. David showed that if, in addition to these assumptions, $f$ satisfies a certain technical condition that we will discuss below, then it is quantitatively bi-Lipschitz on a set of large measure. By verifying his technical condition, David then applied this theorem to show that if an $L$-Lipschitz map $f$ from the unit cube $[0,1]^d$ into $\mathbb{R}^d$ has an image of Lebesgue measure at least $\delta>0$, then $f$ is $M$-bi-Lipschitz on a set of Lebesgue measure $\theta$ in the cube, where $\theta$ and $M$ depend only on $L$ and $\delta$.

Quite different methods were then invented by Jones \cite{Jo88} and David \cite{Da91} to show the result in the case $X=[0,1]^d$ and $Y=\mathbb{R}^D$ equipped with $d$-dimensional Hausdorff measure, where $D\geq d$. In 2009, Schul \cite{Sc09} showed the result in the case where $X=[0,1]^d$ and $Y$ is an arbitrary metric space, again equipped with $d$-dimensional Hausdorff measure. In addition, Meyerson \cite{Me13} used techniques of Jones and David to show the result when $X$ and $Y$ are Carnot groups.

Here we do not use the later methods of \cite{Jo88}, \cite{Da91}, and \cite{Sc09}, but rather the original method of David, which required verifying a certain technical condition on the Lipschitz map and the spaces in question. Originally, this applied only in the case $Y=\mathbb{R}^d$, but later Semmes \cite{Se00} generalized David's theorem to the case of arbitrary target spaces $Y$.

In this paper, we apply this theorem of Semmes and adapt David's original argument to show the ``Lipschitz implies bi-Lipschitz'' result for Lipschitz maps between certain types of abstract manifolds. Our main result is the following theorem. 

\begin{thm}\label{blp}
Let $X$ and $Y$ be Ahlfors $s$-regular, linearly locally contractible, complete, oriented, topological $d$-manifolds, for $s>0$, $d\in \mathbb{N}$. Suppose in addition that $Y$ has $d$-manifold weak tangents.

Suppose $I_0$ is a dyadic $0$-cube in $X$ and $z:I_0\rightarrow Y$ is a Lipschitz map. Then for every $\epsilon>0$, there are measurable subsets $E_1, \dots, E_l \subset I_0$, such that $z|_{E_i}$ is $M$-bi-Lipschitz, and
$$ \left|z\left(I_0 \setminus \bigcup_{i=1}^l  E_i\right)\right| < \epsilon \left|I_0\right|. $$

Here $l$ and $M$ depend only on $\epsilon$, the Lipschitz constant of $z$, the data of $X$, and the space $Y$. 
\end{thm} 

All the relevant definitions will be given in Subsections \ref{definitions} and \ref{dyadiccubes} below.

Note that Theorem \ref{blp} implies in particular the type of result mentioned at the beginning of this paper: if the image of $z$ has positive measure in $Y$, then $z$ is bi-Lipschitz on a set of definite size in $I_0$. This stronger conclusion, in which the domain of the mapping admits a decomposition into pieces on which the mapping is bi-Lipschitz and a ``garbage'' piece of small image, is typical and appears in the works \cite{Jo88}, \cite{Da91}, \cite{Se00}, \cite{Sc09}, and \cite{Me13} mentioned above.

\subsection{Definitions}\label{definitions}
Recall that if $(X,d)$ and $(Y,\rho)$ are metric spaces, then a map $f:X\rightarrow Y$ is called \textit{Lipschitz} if there is a constant $C$ such that
$$ \rho(f(x), f(y)) \leq C d(x,y). $$
The map $f$ is called \textit{bi-Lipschitz}, if there is a constant $C$ such that
$$ C^{-1} d(x,y) \leq \rho(f(x), f(y)) \leq C d(x,y). $$
If we wish to emphasize the constant, we will call such mappings $C$-Lipschitz or $C$-bi-Lipschitz.

The following definition makes sense for general measures, but, following Semmes \cite{Se00}, we will consider only Hausdorff measure $\mathcal{H}^s$.
\begin{definition}\label{AR}
A metric space $(X,d)$ is \textit{Ahlfors $s$-regular} if there is a constant $C_0$ such that for all $x\in X$ and $r\leq\diam X$, we have
$$ C_0^{-1} r^s \leq \mathcal{H}^s(\overline{B}(x,r)) \leq C_0 r^s. $$
\end{definition}

From now on, whenever we speak of a measure we will speak of $s$-dimensional Hausdorff measure in an Ahlfors $s$-regular space. To simplify notation, we therefore always write $|A|$ or, to avoid confusion, $|A|_X$ for the $s$-dimensional Hausdorff measure of a set $A$ in a space $X$.

\begin{definition}\label{LLC}
A metric space $(X,d)$ is called \textit{linearly locally contractible} if there are constants $L, r_0>0$ such that every open ball $B\subset X$ of radius $r<r_0$ is contractible inside a ball with the same center of radius $L r$.  We may abbreviate the condition as LLC or $(L,r_0)$-LLC to emphasize the constants.
\end{definition}

\begin{rmk}
In some contexts, the abbreviation LLC refers to the weaker condition of ``linear local connectivity''. We do not use this condition in this paper.
\end{rmk}

The class of source and target spaces we consider in this paper are complete, oriented topological $d$-manifolds that are Ahlfors $s$-regular and LLC. If $X$ is such a space, the phrase ``the data of $X$'' refers to the collection of constants associated to $X$: the dimensions $d$ and $s$, the constant $C_0$ appearing in the Ahlfors regularity of $X$, and the constants $L$ and $r_0$ appearing in the LLC property of $X$.

There is also an additional constraint on the class of target spaces for which our theorem applies. This requires the notion of convergence of a sequence of pointed metric spaces, which we introduce in Definition \ref{spaceconvergence} below.

\begin{definition}\label{manifoldtangents}
We say a complete metric space $(Y,\rho)$ has $d$-\textit{manifold weak tangents} if the following holds: Whenever ${r_i}$ is a sequence of positive real numbers that is bounded above, $p_i$ are points in $Y$, and $(Y, \frac{1}{r_i} \rho, p_i)$ converges (as in Definition \ref{spaceconvergence}) to a space $(Y_\infty, \rho_\infty, p_\infty)$, then $Y_\infty$ is a topological $d$-manifold.
\end{definition}

\begin{rmk}
Note that Definition \ref{manifoldtangents} includes the assumption that $Y$ itself is a topological $d$-manifold, by taking $r_i=1$ and $p_i=p$ for all $i$.
\end{rmk}

\begin{rmk}\label{manifoldtangentexamples}
While Definitions \ref{AR} and \ref{LLC} are rather standard, Definition \ref{manifoldtangents} is more unusual, and somewhat restrictive. Here are some examples of spaces that satisfy it:
\begin{itemize}
\item The simplest example is $\mathbb{R}^d$, for $d\geq 1$. Indeed, if $Y=\mathbb{R}^d$, then all the pointed metric spaces $(Y, \frac{1}{r_i} \rho, p_i)$ are isometric to $(\mathbb{R}^d,|\cdot|,0)$ by rescaling and translating. Therefore, the limiting space of this sequence is also $\mathbb{R}^d$, which is a topological $d$-manifold.
\item For the same reasons, every Carnot group $G$, equipped with its Carnot-Carath\'eodory metric, has $d$-manifold weak tangents, where $d$ is the topological dimension of $G$.
\item If $X$ is a compact, doubling metric space with $d$-manifold weak tangents, and $Y$ is quasisymmetric to $X$, then $Y$ has $d$-manifold weak tangents. This follows, e.g., from \cite{KL04}, Lemmas 2.4.3 and 2.4.7. (For the definition and properties of quasisymmetric mappings, see \cite{He01}.) 
\item Similarly, if $G$ is a topologically $d$-dimensional Carnot group, and $Y$ is quasisymmetric to $G$, then $Y$ has $d$-manifold weak tangents (even if $Y$ has larger Hausdorff dimension than $G$). This includes all ``snowflaked'' Carnot groups, i.e. metric spaces of the form $(G, \rho^\alpha)$, where $0< \alpha \leq 1$ and $(G,\rho)$ is a Carnot group. 
\item The Cartesian product of two spaces $(X,d_X), (Y, d_Y)$ with $n$- and $m$-manifold weak tangents, respectively, (equipped, e.g., with the metric $d((x,y),(x',y')) = d_X(x,x')+d_Y(y,y')$) has $(n+m)$-manifold weak tangents.
\item Any complete, doubling, linearly locally contractible topological $2$-manifold has $2$-manifold weak tangents. Indeed, by Proposition \ref{manifoldlimit} below, every weak tangent of such a space is a homology $2$-manifold (see Definition \ref{hommfld}), and the only homology $2$-manifolds are topological $2$-manifolds. (See \cite{Br97}, Theorem V.16.32.)
\item Suppose a compact metric space $Z$ has the property that every triple of points can be blown up to a uniformly separated triple by a uniformly quasi-M\"obius map. (This condition was studied by Bonk and Kleiner in \cite{BK02_rigidity} and is satisfied by boundaries of hyperbolic groups equipped with their visual metrics.) Then $Z$ has $d$-manifold weak tangents if and only if $Z$ is itself a topological $d$-manifold. This follows from \cite{BK02_rigidity}, Lemma 5.3. (Note that the definition of a weak tangent given in \cite{BK02_rigidity} is different than ours, in that it requires the sequence of scales $1/r_i$ tend to infinity. However, the proof of Lemma 5.3 in \cite{BK02_rigidity} works the same way without this restriction.)
\end{itemize}
\end{rmk}

\subsection{Dyadic cubes}\label{dyadiccubes}

If $X$ is a complete metric space that is Ahlfors $s$-regular with constant $C_0$, we can equip $X$ with a type of ``dyadic decomposition''. The formulation in \cite{Se00}, Section 2.3, is the easiest to apply here. It says that there exists $j_0\in\mathbb{N}\cup\{\infty\}$ (with $2^{j_0} \leq \diam X < 2^{j_0+1}$ if $X$ is bounded) such that for each $j<j_0$, there exists a partition $\Delta_j$ of $X$ into measurable subsets $Q\in \Delta_j$ such that

\begin{itemize}\label{cubes}
\item $Q\cap Q' = \emptyset$ if $Q,Q'\in \Delta_j$ and $Q\neq Q'$.
\item If $j\leq k < j_0$ and $Q\in \Delta_j, Q'\in \Delta_k$, then either $Q\subseteq Q'$ or $Q\cap Q' = \emptyset$. 
\item $C_0^{-1} 2^j \leq \diam Q \leq C_0 2^j$ and $C_0^{-1} 2^{sj} \leq |Q| \leq C_0 2^{sj}$.
\item For every $j< j_0$, $Q\in \Delta_j$, and $\tau>0$, we have
$$ |\{x\in Q : \dist(x, X\setminus Q) \leq \tau 2^j\}| \leq C_0\tau^{1/C_0} |Q|$$
$$ |\{x\in X\setminus Q : \dist(x, Q) \leq \tau 2^j\}| \leq C_0\tau^{1/C_0} |Q|$$
\end{itemize}

Note that these dyadic cubes are not necessarily closed or open, but merely measurable. They are also disjoint, and do not merely have disjoint interiors. In $\mathbb{R}^d$, one should think of these as analogous to ``half-open'' cubes of the form
$$ \left[ n_1 2^j, (n_1+1) 2^j \right) \times \left[ n_2 2^j, (n_2+1) 2^j )\right) \times \dots \times \left[ n_d 2^j, (n_d+1) 2^j )\right), $$
where $n_i\in\mathbb{Z}$.

It follows from the third and fourth conditions that for every $j<j_0$ and $Q\in \Delta_j$, there exists $x\in Q$ such that
\begin{equation*}\label{quasiround}
B(x, c_0 2^j) \subseteq Q \subseteq B(x, C_0 2^j)
\end{equation*}

All the constants in the cube decomposition depend only on $s$ and the Ahlfors-regularity constant of the space, and so we have denoted the larger cube constant also by $C_0$.

\subsection{Background and results}\label{background}

In \cite{Da88}, condition (9), David introduced the following condition for a Lipschitz map defined on a dyadic cube in an Ahlfors-regular space. Though David gave the condition for maps into $\mathbb{R}^d$, in \cite{Se00}, Condition 9.1, Semmes re-formulated David's condition for arbitrary target spaces. This is the formulation we give here. Recall that $|\cdot|$ denotes $s$-dimensional Hausdorff measure.

\begin{definition}\label{davidscondition}
Let $(X,d)$ be an Ahlfors $s$-regular metric space with a system of dyadic cubes as above. Let $(Y,\rho)$ be a metric space. Let $I_0$ be a $0$-cube in $X$, and $z:I_0\rightarrow Y$ be a Lipschitz map. We will say that $z$ satisfies \textit{David's condition} on $I_0$ if the following holds:

For every $\lambda, \gamma > 0$, there exist $\Lambda, \eta > 0$ such that, for every $x\in I_0$ and $j<j_0$, if $T$ is the union of all $j$-cubes intersecting $B(x,\Lambda 2^j)$, and if $T\subseteq I_0$ and $|z(T)|\geq \gamma |T|$, then either:
\begin{enumerate}[(i)]
\item \label{containsball} $z(T) \supseteq B(z(x), \lambda 2^{j})$, or 
\item \label{expandscube} there is a $j$-cube $R\subset T$ such that
$$ |z(R)|/|R| \geq (1+2\eta)|z(T)|/|T|$$
\end{enumerate}
\end{definition}

As in Theorem \ref{blp}, it is convenient to phrase David's condition for $0$-cubes because that is how we will use it, although it makes sense for cubes of all sizes. Note that, given the definition of our cubes in Subsection \ref{dyadiccubes}, a space may contain no $0$-cubes at all, but one can always create some by rescaling the space and relabeling the levels of the cubes.

In essence, David's condition says the following. At every location and scale within $I_0$, if the map $z$ does not collapse the measure of a ball too much, then one of two things must happen: either (\ref{containsball}) the image of this ball contains a ball of definite size (centered at the image of its center), or (\ref{expandscube}) some sub-cube of this ball is expanded by a larger factor than the ball itself. The upshot of (\ref{containsball}) is that the map $z$ does not ``fold'' at this location and scale.

To take a concrete example, suppose $I_0=[0,1]^2\subset \mathbb{R}^2$ and $z$ is the map
$$z(x,y) = \left(\left|x-\frac{1}{2}\right|,y\right),$$
which folds the square in half along its central vertical axis. If $T$ is well away from the folding line $\{x=\frac{1}{2}\}$, then $z$ essentially acts isometrically on $T$ and so condition (\ref{containsball}) of David's condition holds. If $T$ is centered on the folding line, then $|z(T)|/|T|=1/2$ and (\ref{containsball}) fails, but some sub-square $R$ of $T$ to the left or right of the folding line satisfies $|z(R)|/|R| = 1$, so (\ref{expandscube}) holds.

Theorem 10.1 of \cite{Se00}, which is a generalization of Theorem 1 of \cite{Da88}, says the following.

\begin{thm}[\cite{Se00}, Theorem 10.1]\label{davidtheorem}
Let $(X,d)$ be an Ahlfors $s$-regular metric space with a system of dyadic cubes as above. Let $(Y,\rho)$ be an arbitrary metric space. Let $I_0$ be a $0$-cube in $X$, and $z:I_0\rightarrow Y$ be a Lipschitz map. Suppose that $z$ satisfies David's condition on $I_0$. 

Then for every $\epsilon>0$, there are measurable subsets $E_1, \dots, E_l \subset I_0$, such that $z|_{E_i}$ is $M$-bi-Lipschitz, and
$$ \left|z\left(I_0 \setminus \bigcup_{i=1}^l  E_i\right)\right| < \epsilon \left|I_0\right|.$$

The constants $l$ and $M$ depend only on $\epsilon$, the constants associated to the Ahlfors-regularity of the space $X$, the Lipschitz constant of $z$, and the numbers $\Lambda$ and $\eta$ from David's condition (for $\lambda=1$ and $\gamma$ depending only on $\epsilon$ and the Lipschitz constant of $z$.)
\end{thm}

We will apply Theorem \ref{davidtheorem} and a modification of the proof of Theorem 2 of \cite{Da88} to prove Theorem \ref{blp}. It is worth noting that, in Theorem \ref{blp}, the condition that $Y$ is Ahlfors $s$-regular can be relaxed to the condition that $Y$ is doubling and satisfies the upper mass bound
$$ \mathcal{H}^s(\overline{B}_Y(x,r)) \leq C_0 r^s. $$
It is only this half of the Ahlfors regularity of $Y$ that is used in the proof.

On the other hand, the fact that $X$ and $Y$ are have the same topological dimension $d$ is crucial in the setting of Theorem \ref{blp}. In Proposition \ref{spacefilling} below, we will give a counterexample to Theorem \ref{blp} in which $X$ and $Y$ satisfy all the assumptions of the theorem, except that they are manifolds of different topological dimensions.

A few further remarks on the statement of Theorem \ref{blp} are in order.

\begin{rmk}\label{dependence}
That Theorem \ref{blp} gives dependence of constants on the space $Y$ (and not just its data) is a consequence of our compactness style of proof. However, the proof of Theorem \ref{blp} can be modified slightly to reduce the dependence on $Y$ in the following manner. Let $Y$ be a complete, oriented $d$-manifold that is LLC, Ahlfors $s$-regular, and has $d$-manifold weak tangents. Suppose that $Y'$ is LLC, Ahlfors $s$-regular, and is $\eta$-quasisymmetric to $Y$, by a quasisymmetry that maps balls in $Y'$ of radius $1$ to sets of uniformly bounded diameter. Then Theorem \ref{blp} holds for maps $z:X\rightarrow Y'$ with constants depending only on the space $Y$, the data of $Y'$, and the quasisymmetry function $\eta$ (as well as the data of $z$ and $X$).

In particular, if $\xi\geq \xi_0>0$, then the theorem holds for target space $Y' = (Y,\xi\rho)$ with $l,M$ depending only on $Y$ and $\xi_0$ (as well as on $\epsilon$ and the data of $X$ and $z$), and not on $\xi$ itself. That is because this rescaling is quasisymmetric (with $\eta(t)=t$) and does not alter the data of $Y$, other than changing the contractibility radius $r_0$ to $r_0/\xi_0$.
\end{rmk}

\begin{rmk}\label{allcubes} 
We have phrased Theorem \ref{blp} for $0$-cubes to parallel Theorem 2 of \cite{Da88}. However, it is easy to see that the following statement also holds:

Suppose $j_1< j_0$, $Q_0$ is a dyadic $j$-cube in $X$, $j\leq j_1$, and $z:Q_0\rightarrow Y$ is Lipschitz.
Then the conclusion of Theorem \ref{blp} holds for $z$ on $Q_0$, i.e. for every $\epsilon>0$, there are measurable subsets $E_1, \dots, E_l \subset Q_0$, such that $z|_{E_i}$ is $M$-bi-Lipschitz, and
$$ \left|z\left(I_0 \setminus \bigcup_{i=1}^l  E_i\right)\right| < \epsilon \left|I_0\right|$$
Here $l$ and $M$ depend only on $\epsilon$, the Lipschitz constant of $z$, $j_1$, the space $Y$, and the data of $X$.

Indeed, if $Q_0$ is an $j$-cube for $j\leq j_1$, one need only apply Theorem \ref{blp} to the rescaled spaces $(X, 2^{-j} d)$ and $(Y, 2^{-j} \rho)$, and the same Lipschitz map $z$, relabeling the cubes so that $Q_0$ is a $0$-cube. The rescaled spaces $(X, 2^{-j} d)$ and $(Y, 2^{-j}\rho)$ have the same data as $X$ and $Y$, except that their LLC radii $r_0$ must be replaced by $2^{-j_1}r_0$. So we can apply Theorem \ref{blp} and Remark \ref{dependence} to obtain this result.
\end{rmk}

David proved Theorem \ref{blp} in the case $X=Y=\mathbb{R}^d$ (see \cite{Da88}, Theorem 2). In doing so, he used a compactness argument to verify a modified version of what we have called David's condition. The general idea is the following: Consider a sequence of counterexample maps $z_k$, which in the case of $\mathbb{R}^d$ may all be defined on the unit cube, that fail both conditions of Definition \ref{davidscondition} with increasingly worse constants as $k\rightarrow\infty$. Extract a sub-limit $z$, and by a careful argument show that $z$ has constant Jacobian. Because $z$ is in addition Lipschitz, it is a quasi-regular mapping, and a theorem of Reshetnyak implies that it is an open mapping. A degree argument then shows that, for $k$ large, the image of the maps $z_k$ must contain a fixed size ball around $z_k(0)$, with radius independent of $k$. For $k$ large, this contradicts the assumption that the maps $z_k$ fail the first condition of Definition \ref{davidscondition}.

In our setting, we follow a similar approach. The compactness argument of \cite{Da88} is modified to be a Gromov-Hausdorff compactness argument; to make the degree theory work in this setting we require some results on the Gromov-Hausdorff limits of locally contractible generalized manifolds: see Section \ref{maintools} below. In addition, the theory of quasi-regular mappings and the result of Reshetnyak are not available to us. They are replaced by a topological theorem of Bonk and Kleiner (Theorem \ref{bm} below) on mappings of bounded multiplicity.

A completely different method for verifying David's condition in some situations is a type of detailed homotopy argument, as in \cite{DS00}, Chapter 9. This approach allows for much weaker topological assumptions on $X$, but it seems to rely on having $s=d$, $Y=\mathbb{R}^d$, and $X$ embedded in some Euclidean space.

Even under the assumptions $s=d$ and $Y=\mathbb{R}^d$, Theorem \ref{blp} appears to be new if $X$ is not a subset of some Euclidean space. This observation has a consequence for the geometry of the space $X$. The following concept has many definitions, but the one we give is most natural for abstract metric spaces.

\begin{definition}
An Ahlfors $d$-regular space $X$ is called \textit{uniformly rectifiable} if there exist constants $\alpha>1$ and $0<\beta\leq 1$ such that for every open ball $B$ in $X$, there is a subset $E\subset B$ with $|E|\geq \beta |B|$ and an $\alpha$-bi-Lipschitz map $f:E \rightarrow \mathbb{R}^d$.

We will call $X$ \textit{locally uniformly rectifiable} if for every $r>0$, there exist constants $\alpha$ and $\beta$, depending on $r$, such that for every open ball $B$ in $X$ of radius less than $r$, there is a subset $E\subset B$ with $|E|\geq \beta |B|$ and an $\alpha$-bi-Lipschitz map $f:E \rightarrow \mathbb{R}^d$.
\end{definition}

We can apply Theorem \ref{blp} and a theorem of Semmes \cite{Se96} to show that some abstract manifolds are uniformly rectifiable. Note that in this case we require that the Ahlfors regularity dimension and the topological dimension of $X$ coincide. Snowflaked metric spaces such as $(\mathbb{R}^n, |\cdot|^{1/2})$ provide counterexamples in the absence of this assumption.

\begin{thm}\label{rect}
An Ahlfors $d$-regular, LLC, complete, oriented topological $d$-manifold is locally uniformly rectifiable. The local uniform rectifiability constants $\alpha$ and $\beta$ depend on the scale $r$ and otherwise only on the data of the space.

In particular, a compact, Ahlfors $d$-regular, LLC, oriented topological $d$-manifold is uniformly rectifiable, with constants depending only on $d, C_0, L,$ and $r_0/\diam(X)$.
\end{thm} 

If $X$ admits a bi-Lipschitz embedding into some Euclidean space, then Theorem \ref{rect} follows from work of David and Semmes in \cite{DS00}. However, examples of Semmes \cite{Se96_bilip} and Laakso \cite{La02} show that such an embedding need not always exist.

\subsection*{Acknowledgments}
The author wishes to thank Mario Bonk for a vast number of illuminating conversations, on the contents of this paper and on mathematics in general. He is also grateful for helpful discussions with Peter Petersen, Raanan Schul, and Matthew Badger.

\section{The main tools}\label{maintools}

In this section, we introduce the main concepts and results used in the proof of Theorem \ref{blp}.

\subsection{Convergence of metric spaces}
We will use the notion of convergence of ``mapping packages'', a version of Gromov-Hausdorff convergence, that is described in Chapter 8 of \cite{DS97}. All material in this sub-section is from that source. A brief exposition of this material is also given in \cite{Ke03}.

While the notation in this set-up is a bit more cumbersome than for other definitions of Gromov-Hausdorff convergence, the detailed results of \cite{DS97} make it very flexible for discussing simultaneous convergence of metric spaces and mappings. 

\begin{definition}\label{sets}
We say that a sequence $\{F_j\}$ of non-empty closed subsets of some Euclidean space $\mathbb{R}^N$ converges to a non-empty closed set $F\subseteq\mathbb{R}^N$ if
$$ \lim_{j\rightarrow\infty} \sup_{x\in F_j\cap B(0,R)} \dist(x,F) = 0 $$
and
$$  \lim_{j\rightarrow\infty} \sup_{y\in F\cap B(0,R)} \dist(y,F_j) = 0 $$
for all $R>0$. 
\end{definition}

This convergence is stable under taking products, in the sense that if $\{F_j\}$ converges to $F$ in $\mathbb{R}^N$ and $\{G_j\}$ converges to $G$ in $\mathbb{R}^M$, then
$\{F_j\times G_j\}$ converges to $F\times G$ in $\mathbb{R}^{N+M}$.

\begin{definition}\label{maps}
Suppose $\{F_j\}$ is a sequence of closed sets converging to a closed set $F$ in $\mathbb{R}^N$ as in the previous definition. Let $Y$ be a metric space and $\phi_j:F_j\rightarrow Y$, $\phi:F\rightarrow Y$ be mappings. We say that $\{\phi_j\}$ converges to $\phi$ if for each sequence $\{x_j\}$ in $\mathbb{R}^N$ such that $x_j\in F_j$ for all $j$ and $x_j\rightarrow x\in F$, we have that
$$ \lim_{j\rightarrow\infty} \phi_j(x_j) = \phi(x). $$
\end{definition}

A \textit{pointed metric space} is a triple $(X, d, p)$, where $(X,d)$ is a metric space and $p$ is a point in $X$. All metric spaces that we consider are complete and doubling.

\begin{definition}\label{spaceconvergence}
A sequence of pointed metric spaces $\{(X_j, d_j, p_j)\}$ converges to a pointed metric space $(X,d,p)$ if the following conditions hold. There exists $\alpha\in(0,1]$, $N\in\mathbb{N}$, and $L$-bi-Lipschitz embeddings $f_j:(X_j, d_j^\alpha)\rightarrow\mathbb{R}^N$, $f:(X,d^\alpha)\rightarrow\mathbb{R}^N$ with $f_j(p_j) = f(p) = 0$ for all $j$. Furthermore, we require that $f_j(X_j)$ converge to $f(X)$ in the sense of Definition \ref{sets}, and that the real-valued functions $d_j(f_j^{-1}(x), f_j^{-1}(y))$ defined on $f_j(X_j) \times f_j(X_j)$ converge to $d(f^{-1}(x), f^{-1}(y))$ on $f(X) \times f(X)$ in the sense of Definition \ref{maps}.
\end{definition}

We only use Definition \ref{spaceconvergence} when the metric spaces $\{(X_j, d_j)\}$ and $(X,d)$ are uniformly doubling. In that case, the embeddings $f_j$ and $f$ can always be found, by Assouad's embedding theorem (see \cite{He01}, Theorem 12.2).

\begin{definition}\label{mappingpackage}
A mapping package consists of a pair of pointed metric spaces $(M, d_M, p)$ and $(N, d_N, q)$ as well as a mapping $g:M\rightarrow N$ such that $g(p)=q$.
\end{definition}

\begin{definition}\label{mappingpackageconvergence}
A sequence of mapping packages $\{((X_j, d_j, p_j), (Y_j, \rho_j, q_j), h_j)\}$ is said to converge to another mapping package $((X, d, p), (Y, \rho, q), h)$ if the following conditions hold. The sequences $\{(X_j, d_j, p_j)\}$ and $\{(Y_j, \rho_j, q_j\}$ converge to $(X,d,p)$ and $(Y,\rho, q)$, respectively, in the sense of Definition \ref{spaceconvergence}. Furthermore, the maps $g_j \circ h_j \circ f_j^{-1}$ converge to $g \circ h \circ f^{-1}$ in the sense of Definition \ref{maps}, where $f_j, g_j, f, g$ are the embeddings of Definition \ref{spaceconvergence}.
\end{definition}

The following proposition is a special case of Lemma 8.22 of \cite{DS97}.

\begin{prop}\label{sublimit}
Let $\{((X_j, d_j, p_j), (Y_j, \rho_j, q_j), h_j)\}$ be a sequence of mapping packages, in which all the metric spaces are complete and uniformly doubling, and in which the maps $h_j$ are uniformly Lipschitz and satisfy $h_j(p_j)=q_j$. Then there exists a mapping package $((X, d, p), (Y, \rho, q), h)$ that is the limit of a subsequence of $\{((X_j, d_j, p_j), (Y_j, \rho_j, q_j), h_j)\}$.
\end{prop}

We will now describe some consequences of the convergence of a sequence of mapping packages, which are Lemmas 8.11 and 8.19 of \cite{DS97}.

\begin{prop}\label{almostisos1}
Suppose a sequence of pointed metric spaces $\{(X_k, d_k, p_k)\}$ converges to the pointed metric space $(X,d,p)$, in the sense of Definition \ref{spaceconvergence}.

Then there exist (not necessarily continuous) mappings $\phi_k:X\rightarrow X_k$ and $\psi_k: X_k \rightarrow X$ such that:
\begin{itemize}
\item For all $k$, $\phi_k(p) = p_k$ and $\psi_k(p_k) = p$.

\item For all $R>0$,
$$ \lim_{k\rightarrow\infty} \sup\{d_X(\psi_k(\phi_k(x), x) : x\in B_X(p, R) \} = 0$$
and
$$ \lim_{k\rightarrow\infty} \sup\{d_{X_k}(\phi_k(\psi_k(x), x) : x\in B_{X_k}(p_k, R) \} = 0.$$

\item For all $R>0$,
$$ \lim_{k\rightarrow\infty} \sup\{|d_{X_k}(\phi_k(x) , \phi_k(y)) - d_X(x,y)| : x,y\in B_X(p, R) \} = 0$$
and
$$ \lim_{k\rightarrow\infty} \sup\{|d_{X}(\psi_k(x) , \psi_k(y)) - d_{X_k}(x,y)| : x,y\in B_{X_k}(p, R) \} = 0.$$
\end{itemize}
\end{prop}

\begin{prop}\label{almostisos}
Suppose a sequence of mappings packages $\{((X_k, d_k, p_k), (Y_k, \rho_k, q_k), h_k)\}$ converges to a mapping package  $((X, d, p), (Y,\rho,q),h)$, where the mappings $h_k$ are uniformly Lipschitz and satisfy $h_k(p_k)=q_k$. Then there exist (not necessarily continuous) mappings $\phi_k:X\rightarrow X_k$ and $\psi_k: X_k \rightarrow X$ satisfying exactly the conditions of Proposition \ref{almostisos1}, and mappings $\sigma_k:Y\rightarrow Y_k$ and $\tau_k: Y_k \rightarrow Y$ satisfying the analogous properties of Proposition \ref{almostisos1}, such that in addition we have the following:

For all $x\in X$, 
$$\lim_{k\rightarrow\infty} \tau_k(h_k(\phi_k(x))) = h(x)$$
and this convergence is uniform on bounded subsets of $X$.
\end{prop}

We will be interested in mapping packages in which the mappings $h_k$ are defined only on subsets of the source spaces $X_k$. For this, we need the following fact, which is Lemma 8.32 of \cite{DS97}.

\begin{lemma}\label{subsetconvergence}
Suppose that $\{(X_k, d_k, p_k)\}$ is a sequence of pointed metric spaces that converges to the pointed metric space $\{(X, d, p)\}$ in the sense of Definition \ref{spaceconvergence}. Let $\{F_k\}$ be a sequence of nonempty closed sets with $F_k\subset X_k$ for each $k$. Suppose that 
$$ \sup_k d_k(F_k, p_k) < \infty. $$
Then we can pass to a subsequence to get convergence to a nonempty closed subset $F$ of $X$.
\end{lemma}

We make one final remark in this Subsection, which is Lemma 8.29 of \cite{DS97}.

\begin{lemma}\label{ARlimit}
Let the pointed metric spaces $(X_j, d_j, p_j)$ converge to $(X,d,p)$ in the sense of Definition \ref{spaceconvergence}. Suppose that $(X_j, d_j)$ are Ahlfors $s$-regular, with Ahlfors regularity constant uniformly bounded (see Definition \ref{AR}). Then $(X,d)$ is Ahlfors $s$-regular, with constant controlled by the Ahlfors regularity constants of the spaces $(X_j, d_j)$. 
\end{lemma}

\subsection{Convergence of LLC spaces}
Here we state some results that apply to the convergence of metric spaces (in the sense of the previous section) when those metric spaces also happen to be linearly locally contractible. The main goals are to show that a convergent sequence of uniformly LLC spaces has an LLC limit (essentially a result of Borsuk \cite{Bo55}), and to describe a result that improves Proposition \ref{almostisos} in this context.

The following basic fact about LLC spaces will be used a number of times.

\begin{lemma}\label{ballsareconnected}
Let $X$ be a $(L,r_0)$-LLC space. Fix $x\in X$ and $r<r_0$. Then there is a connected open set $U$ satisfying
$$ B(x,r/(2L)) \subset U \subset B(x,r). $$
\end{lemma}
\begin{proof}

Consider a point $y\in X$ and radius $0<r<r_0$. Let $H:B(y,r/(2L)) \times [0,1] \rightarrow B(x,r/2)$ be a homotopy contracting $B(y,r/(2L))$ to a point. Define
$$E(y,r) = H(B(y,r/(2L)) \times [0,1]).$$
Then $E(y,r)$ is a connected subset of $B(y,r/2)$ containing $B(y,r/(2L))$.

Let $E_0= E(x,r)$. For $i\in\mathbb{N}$, inductively define sets
$$ E_i = \bigcup_{y\in E_{i-1}} E(y,2^{-i}r).$$
By induction, each set $E_i$ is connected. In addition, for each $i$ we have the relation
\begin{equation}\label{int}
E_i \subset \interior E_{i+1}.
\end{equation}
Now let
$$ U = \bigcup_{i=1}^\infty E_i. $$
Then, as the union of connected sets that all contain the point $x$, $U$ is connected. In addition, by (\ref{int}) $U$ is open: if $x\in U$, then, for some $i$,
$$ x\in E_i \subset \interior E_{i+1} \subset \interior U. $$
Finally, if $y\in E_i \subset U$, then
$$ d(x,y) \leq \left(2^{-(i+1)} + 2^{-i} + \dots + 2^{-1}\right)r < r.$$
Thus, $U$ is a connected open set, $U\subset B(x,r)$, and $U\supseteq E_0\supseteq B(x,r/(2L))$.
\end{proof}

The following is our main lemma about convergence of uniformly LLC sets.

\begin{lemma}\label{LLClimitset}
Let $F_k$ be a sequence of closed sets in some Euclidean space $\mathbb{R}^N$ that are each $(L,r_0)$-LLC (as spaces equipped with the induced Euclidean metric). Suppose that $F_k\rightarrow F$ in the sense of Definition \ref{sets}. Then $F$ is LLC, with constants depending only on $L$ and $r_0$.
\end{lemma}
In the case of compact sets converging in the usual Hausdorff metric, Lemma \ref{LLClimitset} is due to Borsuk \cite{Bo55}. A similar localized version of the result was noted in \cite{HK11}. Here we provide a proof, following the method of Borsuk.

The proof is somewhat technical, though the main idea is not difficult: For subsets of Euclidean space, the LLC property for a set $E$ implies the existence of a retraction to $E$, from an open neighborhood of $E$ of fixed size, that moves points by an amount proportional to their distance from $E$. We use the existence of these retractions on the limiting sets $F_k$ to construct a retraction onto the limit $F$. This retraction can then be used to show that $F$ is LLC. Because the convergence is local, there are some minor technical complications.

\begin{proof}[Proof of Lemma \ref{LLClimitset}]
For a set $E\subseteq\mathbb{R}^N$, let $U_\epsilon(E)$ denote the open $\epsilon$-neighborhood of $E$. Let $B_R = B(0,R) \subset \mathbb{R}^N$.

We note first that the LLC property implies that there exist constants $c<1$ and $C=c^{-1}>1$ such that each $F_k$ admits a continuous retraction $r_k:U_c(F_k)\rightarrow F_k$ satisfying
\begin{equation}\label{retraction}
|r_k(x) - x| \leq C \dist(x, F_k)
\end{equation}
for $x\in U_c(F_k)$.
The proof of this can be found in Section 13 of \cite{Bo55} (and does not require compactness of the sets).

Fix a ball $B=B(p,r)\cap F$ for $p\in F$ and $r<r_0/4$. Fix $R>\max\{4Lr, 12C\}$ large enough so that $\overline{B}\subset B_R$.

By passing to a subsequence, we may without loss of generality assume that, for all $k$,
$$ \sup\{ \dist(x, F_k) : x\in F\cap B_{10R}\} < c/4 $$
$$ \sup\{ \dist(x, F) : x\in F_k\cap B_{10R}\} < c/4 $$

It follows that
$$U:= \bigcap_{k=1}^\infty U_c(F_k)$$
contains a $c/2$-neighborhood of $B_{9R}\cap F$ as well as of $\displaystyle\bigcup_{k=1}^\infty (B_{9R} \cap F_k)$.

For $k\in\mathbb{N}$, fix decreasing sequences
\begin{equation}\label{eta}
 \eta_k = c 4^{-k},
\end{equation}
\begin{equation}\label{eta'}
 \eta'_k = c 4^{-k}/3.
\end{equation}

We may now pass to a further subsequence of our sets on which we assume that
\begin{equation}\label{close1}
\sup\{ \dist(x, F_k) : x\in F\cap B_{9R}\} < \eta'_k/8,
\end{equation}
\begin{equation}\label{close2}
\sup\{ \dist(x, F) : x\in F_k\cap B_{9R}\} < \eta'_k/8.
\end{equation}

Let $U_k = U_{\eta_k}(F_k)$ and $V_k = U_{\eta'_k}(F_k)$. Then, if $x\in U_{k+1} \cap B_{7R}$, we have, by (\ref{eta}), (\ref{eta'}), (\ref{close1}), and (\ref{close2}), that
$$ \dist(x,F_k\cap B_{8R}) < \eta'_k. $$
Therefore, for every $R'\leq 7R$,
\begin{equation}\label{inclusions}
(U_{k+1} \cap B_{R'}) \subset (V_{k} \cap B_{R'}) \subset (\overline{V}_{k} \cap B_{R'}) \subset (U_{k} \cap B_{R'}) \subset  (U \cap B_{R'})
\end{equation}

We will now inductively construct a new sequence of retractions $s_k:U\cap B_{5R} \rightarrow F_k$ by modifying the maps $r_k$.

Let $s_1 = r_1$. Suppose that $s_k$ has already been defined and in addition satisfies $s_k = r_k$ on $V_k \cap B_{5R}$. Let $f:U\rightarrow \mathbb{R}$ be a continuous function that is $0$ on $U\setminus U_{k+1}$ and $1$ on $V_{k+1}$. For $x\in U\cap B_{5R}$, define
$$ s_{k+1}(x) = r_{k+1}((1-f(x))s_k(x) + f(x)x). $$

We first check that $s_{k+1}$ is well-defined, i.e. that for $x\in U\cap B_{5R}$, the point $(1-f(x))s_k(x) + f(x)x$ is in $U$. If $x\in U\setminus U_{k+1}$, then $(1-f(x))s_k(x) + f(x)x = s_k(x)\in F_k \in U$, so $s_{k+1}$ is well-defined. In the case $x\in U_{k+1}$, we have by (\ref{inclusions}) that $x\in V_k$. By our inductive assumption that $s_k = r_k$ on $V_k \cap B_{5r}$, we get
$$ |x-s_k(x)| = |x-r_k(x)| \leq C\eta'_k < c. $$
Thus, every point on the line segment from $x$ to $s_k(x)$ is in the $c$-neighborhood of $F_k$ and so is in $U$.

That $s_{k+1}$ is the identity on points of $F_{k+1}$ follows from the fact that, by definition, $s_{k+1}=r_{k+1}$ on $F_{k+1}$.

We now make the following claim: If $x\in U\cap B_{5R}$ and $s_k(x)\in B_{6R}$, then 
\begin{equation}\label{previous}
|s_{k+1}(x) - s_k(x)| < 3C 4^{-k}.
\end{equation}
To prove this, we consider three cases.
\begin{enumerate}[(i)]
\item The case $x\in V_{k+1}$:

In this case, using (\ref{retraction}) and the definitions of $s_k, s_{k+1}$, we get
$$|s_{k+1}(x) - s_k(x)| = |r_{k+1}(x) - r_k(x)| \leq |r_{k+1}(x) - x| +|x - r_k(x)| \leq C(\eta'_{k+1} + \eta'_k) < 3C 4^{-k}.$$

\item The case $x\in U\setminus U_{k+1}$:

In this case, $s_{k+1}(x) = r_{k+1}(s_k(x))$. By assumption, $s_k(x)\in F_k \cap B_{6R}$ and therefore $\dist(s_k(x), F_{k+1}) < \eta'_k/4$ by (\ref{close2}). Therefore, by (\ref{retraction}),
$$ |s_{k+1}(x) - s_k(x)| = | r_{k+1}(s_k(x)) - s_k(x)| \leq C\eta'_k/4 < 3C 4^{-k}.$$

\item The case $x\in U_{k+1} \setminus V_{k+1}$:

Note that $x\in U_{k+1}\cap B_{5R} \subset V_{k} \cap B_{5R}$, so $s_k(x) = r_k(x)$. Let
$$y = (1-f(x))s_k(x) + f(x)x,$$
which is on the line segment $L$ joining $x$ to $s_k(x)=r_k(x)$. The diameter of $L$ is therefore $|x-r_k(x)|\leq C\eta'_k$, by (\ref{retraction}) and the fact that $x\in V_k$.

In addition, because $x\in U_{k+1}$, we have $\dist(x,F_{k+1})<\eta_k$. 

From these calculations, it follows that
$$ \dist(y,F_{k+1}) \leq \dist(x,F_{k+1}) + \diam(L) \leq \eta_k + C\eta'_k, $$
and therefore, by (\ref{retraction}), that
$$ |s_{k+1}(x) - x| = |r_{k+1}(y) - x| \leq |r_{k+1}(y) - y| + |y-x| \leq C(\eta_k + C\eta'_k) + C\eta'_k \leq 2C4^{-k}.$$
From this, we see that 
$$ |s_{k+1}(x) - s_k(x)| \leq |s_{k+1}(x) - x| + |x-r_k(x)| < 2C 4^{-k} + \eta_k < 3C4^{-k}. $$
\end{enumerate}

This concludes the proof of the claim that $|s_{k+1}(x) - s_k(x)|<3C 4^{-k}$ if $x\in U\cap B_{5R}$ and $s_k(x)\in B_{6R}$.

Now note that
$$|s_1(x) - x| = |r_1(x) - x| \leq Cc = 1.$$
Therefore $r_1(x) \in B_{5.5 R}$. Because $\displaystyle \sum_{k=0}^\infty (3C4^{-k}) \leq 6C < R/2$, it follows from the above claim that $s_k(x) \in B_{6R}$ for all $k$, and therefore that
$$ |s_{k+1}(x) - s_k(x)| < 3C 4^{-k} $$
for all $k$.

It follows immediately from this and from (\ref{retraction}) that $s_k|_{U\cap B_{5R}}$ converge uniformly to a map
$$s:U\cap B_{5R} \rightarrow F\cap B_{6R}$$
that is the identity on $F\cap B_{5R}$.

Note that the map $s$ is indeed the identity on $F$: if $x\in F\cap B_{5R}$, then by (\ref{close1}) and the definition of $s_k$ we see that $s_k(x) = r_k(x)$. It follows that
$$ |s(x) - x| = \lim_{k\rightarrow\infty} |s_k(x) - x| = \lim_{k\rightarrow\infty} |r_k(x) - x| \leq C\lim_{k\rightarrow\infty} \dist(x,F_k) = 0. $$

To finish the proof of the lemma, recall our fixed ball $B=B(p,r)\cap F$ in $F\cap B_{R}$. The map $s$, when restricted to $F\cap\overline{B}_{4R}$, is the identity. Therefore, for every positive number $\eta<r$ sufficiently small, there is a neighborhood $V\subset (U\cap B_{5R})$ of $F\cap\overline{B}_{4R}$ such that
$$ x\in V \Rightarrow |s(x) - x|<\eta.$$

We may now choose $k$ large so that $|s_k(x) - s(x)|<\eta$ for all $x\in U\cap B_{5R}$ (by uniform convergence) and in addition so that
$$ F_k \cap B_{3R} \subset V. $$

Now we contract $B$ in the following manner. First, consider the homotopy
$$ h(x,t) = (1-t)x + ts_k(x)$$
for $x\in B$ and $t\in[0,1]$. Because $|s_k(x) -x|=|s_k(x) - s(x)|<\eta$, we have $h(B\times [0,1])\subset B_{3R}$. In addition, $h$ deforms $B$ onto a set $E\subset F_k \cap B_{3R}$ of diameter no more than $2r+2\eta$. By our choices of $r$ and $\eta$, $2r+2\eta<4r<r_0$, and therefore $E$ is contractible inside a set $E'\subset F_k\cap B_{3R}$ of diameter $L(2r+2\eta)$.

Let $g:B\times [0,1]\rightarrow E'\subset (F_k\cap B_{3R})$ denote the homotopy of $B$ onto a point that first deforms by $h$ and then by the contraction in $F_k$. Then $s \circ g$ is a contraction of $B$ to a point within the set $s(E')\subset F$, which has diameter no more than $L(2r+2\eta) + 2\eta$.

In summary, if we recall that $\eta<r$, we have shown that the ball $B=B(p,r) \cap F$ is contractible within the ball $B' = B(p, (4L+2)r) \cap F$ whenever $r<r_0/4$. This completes the proof.
\end{proof}

\begin{lemma}\label{limitLLC}
Suppose the pointed metric spaces $(X_k, d_k, v_k)$ are $(L,r_0)$-LLC and converge to the pointed metric space $(X,d,v)$ in the sense of Definition \ref{spaceconvergence}. Then $(X,d)$ is LLC, with constants depending only on $L, r_0$.
\end{lemma}
\begin{proof}
This follows immediately from Lemma \ref{LLClimitset} and Definition \ref{spaceconvergence}, as the ``snowflake'' transformations of Definition \ref{spaceconvergence} distort the LLC constants in a quantitative way.
\end{proof}

To conclude this section, we give two lemmas which improve Propositions \ref{almostisos1} and \ref{almostisos} in the setting of LLC spaces. They say that if a sequence of mapping packages converges, then the ``almost-isometries'' $\phi_k$ and $\psi_k$ between the limiting spaces and the limit space can be taken to be continuous.

\begin{definition}
For $\eta>0$, we say that continuous maps $f,g:M\rightarrow N$ between metric spaces are \textit{$\eta$-homotopic} if they are homotopic by a homotopy $H:M\times[0,1]\rightarrow N$ such that, for all $x\in M$ and $t\in [0,1]$, we have
$$ d_N (f(x), H(x,t)) < \eta. $$
\end{definition}

Note in particular that if $f$ and $g$ are $\eta$-homotopic, then $d_N(f(x), g(x))<\eta$ for all $x$.

The following is an immediate consequence of Proposition \ref{almostisos} above, combined with Propositions 5.4 and 5.8 of \cite{Se96}. (See also \cite{Pe90}, Section 3, for a cleaner statement in the compact case.) Note that all our spaces are Ahlfors $s$-regular and thus have topological dimension bounded above by $s$, so those results apply.

\begin{lemma}\label{ctsalmostiso}
Suppose the pointed metric spaces $(X_k, d_k, v_k)$ are $(L,r_0)$-LLC, uniformly Ahlfors $s$-regular, and converge to the pointed metric space $(X,d,v)$ in the sense of Definition \ref{spaceconvergence}.

Fix a point $x\in X$ and radius $R>0$. Then there exist continuous mappings $f_k:\overline{B}_X(x,R) \rightarrow X_k$ and $g_k:\overline{B}_{X_k}(f_k(x),R)\rightarrow X$ satisfying the following conditions:
\begin{enumerate}[\normalfont (i)]
\item\label{isom} They almost preserve distances, in the sense that
$$ \lim_{k\rightarrow\infty} \sup\{|d_{X_k}(f_k(p) , f_k(q)) - d_X(p,q)| : p,q\in B_X(x, R) \} = 0$$
and
$$ \lim_{k\rightarrow\infty} \sup\{|d_{X}(g_k(p) , g_k(q)) - d_{X_k}(p,q)| : p,q\in B_{X_k}(f_k(x), R) \} = 0.$$

\item\label{inverse} For every $0<r<R$, we have
\begin{equation*}
 \lim_{k\rightarrow\infty} \inf \{\eta: g_k \circ f_k|_{\overline{B}(x,r)} \text{ is $\eta$-homotopic to the inclusion of } \overline{B}(x,r) \text{ into } B(x,R) \} = 0
\end{equation*}
and
\begin{equation*}
\lim_{k\rightarrow\infty} \inf \{\eta: f_k \circ g_k|_{\overline{B}(f_k(x),r)} \text{ is $\eta$-homotopic to the inclusion of }\overline{B}(f_k(x),r) \text{ into } B(f_k(x),R) \} = 0
\end{equation*}

\item\label{basepoint} If $x$ is the basepoint $v\in X$, then in addition we have
$$ \lim_{k\rightarrow\infty} d_k(f_k(v), v_k) = 0$$
\end{enumerate}
\end{lemma}
\begin{proof}
Take $\eta>0$. We will find, for all $k$ sufficiently large, continuous mappings $f_k$ and $g_k$ as above that preserve distances up to additive error $\eta$ and such that $f_k\circ g_k$ and $g_k\circ f_k$ are $\eta$-homotopic to the appropriate inclusion maps.

Fix numbers $\eta'', \eta'$ sufficiently small, with $0<\eta''<\eta'<\eta$. They will depend only on $\eta$ and the (uniform) data of the spaces $X$, $\{X_k\}$.

By Proposition \ref{almostisos1}, there is an index $k_0\in\mathbb{N}$ such that, for all $k\geq k_0$, the maps $\phi_k:X\rightarrow X_k$ and $\psi_k: X_k\rightarrow X$ preserve distances up to additive error $\eta''$ on $\overline{B}(x,R)$ and $\overline{B}(f_k(x),R)$, respectively. By \cite{Se96}, Proposition 5.4, if $\eta''$ was chosen sufficiently small compared to $\eta'$, then there exist continuous maps $f_k:\overline{B}(x,R)\rightarrow X_k$ and $g_k:\overline{B}(f_k(x),R)\rightarrow X$ such that
\begin{equation}\label{fnearphi}
d_k(f_k(z), \phi_k(z))<\eta' \text{   and   } d(g_k(y), \psi_k(y))<\eta'
\end{equation}
on their respective domains. Part (\ref{isom}) of the lemma follows immediately from this by taking $\eta'<\eta/10$. Part (\ref{basepoint}) also follows, because $\phi_k(v) = v_k$.

Now fix $0<r<R$. By Proposition \ref{almostisos1} we may also assume that, for all $k\geq k_0$, we have
$$ d(\phi_k(\psi_k(x)), x) < \eta' \text{ and } d(\psi_k(\phi_k(x)),x) < \eta',$$
in addition to the properties above.

If $\eta'$ was chosen sufficiently small, then $f_k(\overline{B}(x,r))\subset B(f_k(x),R)$ and so the composition $g_k \circ f_k$ is defined on $\overline{B}(x,r)$. Similarly, the composition $f_k\circ g_k$ is defined on $\overline{B}(f_k(x),r)$. By choosing $\eta''<\eta'/10$ and using equation (\ref{fnearphi}) and the properties of $\phi_k$ and $\psi_k$, we also see that
$$ d(f_k(g_k(x)), x) < 2\eta' \text{ and } d(g_k(f_k(x)),x) < 2\eta'.$$
Therefore, if $\eta'$ was chosen sufficiently small, depending on $\eta$ and the data of the spaces $X$, $\{X_k\}$, Proposition 5.8 of \cite{Se96} implies that
$$g_k \circ f_k|_{\overline{B}(x,r)} \text{ and } f_k \circ g_k|_{\overline{B}(f_k(x),r)}$$
are $\eta$-homotopic to the inclusions
$$ \overline{B}(x,r)\rightarrow B(x,R) \text{ and } \overline{B}(f_k(x),r)\rightarrow B(f_k(x),R).$$

This proves part (\ref{inverse}) of the lemma.
\end{proof}

Propositions 5.4 and 5.8 of \cite{Se96}, on which the proof of the previous lemma is based, are important consequences of the linear local contractibility of the spaces $X$ and $\{X_k\}$. Roughly speaking, they say that if a mapping into an LLC space is ``roughly continuous'' (as the maps $\phi_k$ and $\psi_k$ are), then it is close to a continuous mapping, and if two continuous mappings into an LLC space are close, then they are $\eta$-homotopic. The proofs of these facts use polyhedral approximations of the source space and an induction on the skeleta of the polyhedra. We encourage the reader to look at Semmes's paper \cite{Se96} or Petersen's work \cite{Pe90} for the details.

The following additional fact is an immediate consequence of Lemma \ref{almostisos} and equation (\ref{fnearphi}) above.
\begin{lemma}\label{ctsalmostiso2}
Suppose we have convergence of a sequence of mapping packages
$$ \left((X_k, d_k, p_k), (Y_k, \rho_k, q_k), h_k\right) \rightarrow \left((X, d, p), (Y, \rho, q), h\right) $$
in the sense of Definition \ref{mappingpackageconvergence}. Suppose that all the spaces involved are uniformly Ahlfors $s$-regular and uniformly LLC, and that the mappings $\{h_k\}$ and $h$ are uniformly $C$-Lipschitz and satisfy $h_k(p_k)=q_k$, $h(p)=q$. Then for all $R>0$, there exist continuous mappings
$$f_k:\overline{B}_X(p,R)\rightarrow X_k \text{ and } g_k: \overline{B}_{X_k}(p_k,R) \rightarrow X$$
satisfying exactly the conditions of Lemma \ref{ctsalmostiso}, and continuous mappings
$$\tilde{f}_k:\overline{B}_Y(q,R)\rightarrow Y_k \text{ and } \tilde{g}_k: \overline{B}_{Y_k}(q_k,R) \rightarrow Y$$
satisfying the analogous properties of Lemma \ref{ctsalmostiso}, such that in addition we have that for all $x\in X$, 
$$\lim_{k\rightarrow\infty} \tilde{g}_k(h_k(f_k(x))) = h(x)$$
uniformly on $\overline{B}_X(p,R/2C)$.
\end{lemma}

\subsection{Convergence of manifolds}
Here we state some facts on the convergence of metric spaces that are LLC topological manifolds. Our main goal is to give a proof of Proposition \ref{manifoldlimit} below, which says that the limit of a sequence of uniformly Ahlfors regular, uniformly LLC, topological $d$-manifolds is a homology $d$-manifold (see Definition \ref{hommfld}). This result essentially goes back to Begle \cite{Be44} (see also \cite{GPW90}) and appears to be well-known, but we did not find a modern proof in the literature in the generality necessary here.

Below, $H_*$ denotes singular homology with integer coefficients.

\begin{lemma}\label{poinc}
Let $M$ be an $(L,r_0)$-LLC oriented topological $d$-manifold. Let $v\in M$ and let $K_1 \subset K_2$ be compact sets satisfying $v \in K_1 \subset B(v,r) \subset B(v, 2Lr) \subset K_2 \subset B(v, r_0)$. Then the following facts hold.
\begin{enumerate}[\normalfont (i)]
\item The map $j_{*}: H_p(M,M\setminus K_2) \rightarrow H_p(M, M\setminus K_1)$, induced by inclusion, is trivial if $p\neq d$.\label{poinci}

\item The map $i_{*}: H_d(M, M\setminus K_2) \rightarrow H_d(M, M\setminus \{v\}) \cong \mathbb{Z}$, induced by inclusion, is surjective.\label{poincii}

\item With this notation, we also have $\ker i_* \subseteq \ker j_*$ in the top degree $d$.\label{poinciii}
\end{enumerate}
\end{lemma}
\begin{proof}
By use of the natural duality isomorphisms (\cite{Sp81},Theorem 6.2.17) we obtain the following commutative diagram. Here $\overline{H}$ denotes \v{C}ech cohomology, and all maps in the diagram are the natural maps induced by inclusion.
\begin{equation}
\begin{CD}
H_p(M, M\setminus K_2) @>j_*>> H_p(M, M\setminus K_1) @>k_*>> H_p(M, M\setminus \{v\})\\
@V{\cong}VV @V{\cong}VV @V{\cong}VV\\
\overline{H}^{d-p}(K_2) @>j^*>> \overline{H}^{d-p}(K_1) @>k^*>> \overline{H}^{d-p}(\{v\})
\end{CD}
\end{equation}

If $p\neq d$, then $j^*$ is trivial because $K_1$ is contractible in $K_2$, which proves (\ref{poinci}).

Now let $p=d$. The map $i^* = k^* j^*: \overline{H}^{0}(K_2) \rightarrow  \overline{H}^{0}(\{v\})$ is surjective, as $v\in K_2$, which proves (\ref{poincii}).

Finally, by Lemma \ref{ballsareconnected}, $K_1$ is entirely contained in a connected component $E$ of $K_2$. Therefore, every connected component $E'$ of $K_2$ that does not contain $\{v\}$ is in fact disjoint from $K_1$. It follows that if $i^*\phi = k^*j^* \phi$ is trivial in $\overline{H}^0(\{v\})$, then $j^*\phi$ is already trivial in $\overline{H}^0(K_1)$. This proves claim (\ref{poinciii}).
\end{proof}

We now set up some definitions for the main result of this sub-section. A \textit{Euclidean Neighborhood Retract} (ENR) is a space $X$ which, for every $N\in\mathbb{N}$ and every topological embedding $e:X\rightarrow \mathbb{R}^N$, has the property that $e(X)$ is a retract of some open neighborhood of $e(X)$ in $\mathbb{R}^N$. Every LLC space with finite topological dimension is a Euclidean Neighborhood Retract (see \cite{Hu65}, Theorem V.7.1).

\begin{definition}\label{hommfld}
A space $M$ that is an ENR and that satisfies the condition
$$ H_*(M, M\setminus \{x\}) = H_*(\mathbb{R}^d, \mathbb{R}^d\setminus\{0\}), $$
for all $x\in M$, is called a \textit{homology} $d$\textit{-manifold}.
\end{definition}

\begin{prop}\label{manifoldlimit}
Suppose $(X_k, d_k)$ are uniformly Ahlfors $s$-regular, $(L,r_0)$-LLC oriented topological $d$-manifolds, $v_k\in X_k$, and the sequence of pointed metric spaces $(X_k, d_k, v_k)$ converges to $(X,d,v)$ in the sense of Definition \ref{spaceconvergence}. Then $(X,d)$ is an LLC homology $d$-manifold.
\end{prop}
\begin{proof}

The fact that $(X,d)$ is LLC is Lemma \ref{limitLLC} above. As this statement is quantitative, we will denote the LLC constants of $(X,d)$ also by $(L, r_0)$.

The fact that $X$ is a homology $d$-manifold can be proven by the methods of Begle \cite{Be44}, again as remarked in \cite{HK11}. For convenience, we provide a proof using the tools introduced in this section.

We know that $X$ is Ahlfors $s$-regular, and therefore it has finite Hausdorff dimension and thus finite topological dimension. Because $X$ is also LLC, it is an ENR, as noted above. It now suffices to show that for every $x\in X$, the local integer (singular) homology groups $H_p(X, X\setminus\{x\})$ are isomorphic to $\mathbb{Z}$ if $p=d$ and trivial otherwise.

To set up the proof we need some notation.

Let $L'=4L$. Fix an integer $p\geq 0$, a point $x\in X$, and a radius $R>0$. In addition, for each $k\in\mathbb{N}$, fix continuous maps
$$f_k: B_X(x,R) \rightarrow X_k$$
$$g_k: B_{X_k}(f_k(x),R) \rightarrow X $$
as in Lemma \ref{ctsalmostiso}. These maps have the property that, up to arbitrarily small additive error (decreasing to zero with $k$), they preserve distances and are inverses of each other.

For $n\in\mathbb{N}$, let
$$ F_n = H_p(X, X\setminus \overline{B}(x, (L')^{-n}r_0)) $$
and
$$ G^k_n = H_p(X_k, X_k\setminus \overline{B}(f_k(x), (L')^{-n}r_0)) $$
(Of course these groups depend on $p$, but we will make it clear from context which value of $p$ we take.)

Note that for $m\geq n$ there are natural maps $(i_{n,m})_*: F_n\rightarrow F_{m}$ and $(j^k_{n,m})_*:G^k_n\rightarrow G^k_{m}$ induced by inclusion. 

\begin{claim}\label{directlimits}
We have the direct limits
$$ F_\infty:= \varinjlim F_n \cong H_p(X, X\setminus \{x\})$$
and
$$ G^k_\infty:= \varinjlim G^k_n \cong H_p(X_k, X_k\setminus \{f_k(x)\}) \cong
\begin{cases}
\mathbb{Z}, & \text{if }p=d \\
0, & \text{if }p\neq d
\end{cases} $$
\end{claim}
\begin{proof}[Proof of Claim \ref{directlimits}]
We will show the first direct limit; the proof of the second is identical. The proof follows from standard properties of direct limits and singular homology. There are natural maps $\phi_n:F_n\rightarrow H_p(X, X\setminus \{x\})$ induced by inclusion. To show that $F_\infty \cong H_p(X,X\setminus\{x\})$, we must show two things (see, e.g., \cite{Ma78}, Proposition A.4):
\begin{enumerate}
\item For every $a\in H_p(X,X\setminus\{x\})$, there exists $n\in\mathbb{N}$ and $b\in F_n$ such that $\phi_n(b) = a$. \label{dirlimsurj}
\item If $b\in F_n$ and $\phi_n(B)=0$, then $(i_{n,m})_*(b)=0$ for some $m\geq n$. \label{dirliminj}
\end{enumerate}

To show (\ref{dirlimsurj}), consider $a\in H_p(X,X\setminus\{x\})$. By excision and the fact that singular homology has compact support (see \cite{Sp81}, 4.8.11), $a = j_*(c)$, where $c\in H_p(X, X\setminus U)$ for some open set $U$ containing $x$, and
$$ j_*: H_p(X, X\setminus U) \rightarrow H_p(X, X\setminus\{x\})$$
is the mapping induced by inclusion.

We now choose $n\in\mathbb{N}$ large enough so that $\overline{B}(x, (L')^{-n}r_0) \subset U$. There is a mapping
$$ k_*: H_p(X, X\setminus U) \rightarrow H_p(X, X\setminus\overline{B}(x, (L')^{-n}r_0)) $$
induced by inclusion.

Because all mappings are induced by inclusion, we have $\phi_n k_* = j_*$. Thus, if we let $b=k_*(c)\in H_p(X, X\setminus\overline{B}(x, (L')^{-n}r_0))$, we see that $\phi_n(b) = \phi_n k_* (c) = j_*(c) = a$. This proves part (\ref{dirlimsurj}) of Claim \ref{directlimits}.

To show part (\ref{dirliminj}), suppose that $b\in F_n$ is such that $\phi_n(b) =0\in H_p(X,X\setminus\{x\})$. As before, using the fact that singular homology has compact support, we can write $b=l_*(c)$, where $c\in H_p(X, X\setminus U)$ for some open set $U$ containing $x$, and
$$ l_*: H_p(X, X\setminus U) \rightarrow F_n$$
is the mapping induced by inclusion.

By excision and \cite{Sp81}, Theorem 4.8.13, we see that $i_*(c)=0\in H_p(X, X\setminus V)$, where $V\subset U$ is an open set containing $x$ and
$$ i_*: H_p(X, X\setminus U) \rightarrow H_p(X, X\setminus V)$$
is the mapping induced by inclusion.

We now choose $m\in\mathbb{N}$ large enough so that  $\overline{B}(x, (L')^{-n}r_0) \subset V$. Let 
$$ h_*: H_p(X, X\setminus V) \rightarrow F_m$$
be induced by inclusion.
Again because all mappings commute, we have
$$ (i_{n,m})_* (b) = h_* i_*(c) = h_*(0) = 0 \in F_m.$$
This completes the proof of Claim \ref{directlimits}.
\end{proof}

Let $(i_n)_*:F_n\rightarrow F_\infty$ and $(j^k_n)_*:G^k_n\rightarrow G^k_\infty$ denote the natural inclusion maps.

The excision property of homology and the properties of $f_k$ and $g_k$ allow us to conclude the following: For all $n_0\in\mathbb{N}$, there exists $k_0\in\mathbb{N}$ such that for all $k\geq k_0$ and $n\leq n_0$, there are group homomorphisms $a^k_n:F_n\rightarrow G^k_{n+1}$ and $b^k_n:G^k_n\rightarrow F_{n+1}$ that commute with the inclusion maps above, and that satisfy
$$ b^k_{n+1} a^k_n  = i_{n, n+2} \hspace{0.5in}\text{and}\hspace{0.5in}  a^k_{n+1} b^k_n= j^k_{n, n+2}. $$
Indeed, $a^k_n$ and $b^k_n$ are simply the maps on homology induced by $f_k$ and $g_k$, and so these properties follow from Lemma \ref{ctsalmostiso}. The fact that $a^k_n$ maps into $G^k_{n+1}$ if $n\leq n_0$ $k$ is sufficiently large follows from the fact that $f_k$ preserves distances up to a small additive error, by Lemma \ref{ctsalmostiso}.

In summary, for each $n_0$ there exists a $k$ so that we have the following commutative diagram, in which the diagonal arrows do not exist past column $n_0$:
\begin{equation}\label{diagram}
\begin{tikzcd}[row sep=huge, column sep=huge]
F_1 \arrow{r}{i_{1,2}}\arrow{rd}[near start]{a^k_1} &F_2 \arrow{r}{i_{2,3}}\arrow{rd}[near start]{a^k_2} &F_3 \arrow{r}{i_{3,4}}\arrow{rd}[near start]{a^k_3} &\dots \arrow{r}{i_{n_0-1, n_0}}\arrow{rd}[near start]{a^k_{n_0-1}} &F_{n_0} \arrow{r}{i_{n_0,n_0+1}} &\dots F_\infty \\
G^k_1 \arrow{r}{j^k_{1,2}}\arrow{ru}[near start]{b^k_1} &G^k_2 \arrow{r}{j^k_{2,3}}\arrow{ru}[near start]{b^k_2} &G^k_3 \arrow{r}{j^k_{3,4}}\arrow{ru}[near start]{b^k_3} &\dots\arrow{ru}[near start]{b^k_{n_0-1}} \arrow{r}{j^k_{n_0-1, n_0}} &G^k_{n_0}\arrow{r}{j^k_{n_0,n_0+1}} &\dots G^k_\infty
\end{tikzcd}
\end{equation}

Note that Lemma \ref{poinc} translates to the following information in this setting:
\begin{enumerate}[(i)]
\item \label{i} If $p\neq d$, then for each $k$ and for each $m>n$, the map $j^k_{n,m}: G^k_n \rightarrow G^k_m$ is trivial. 
\item \label{ii} If $p=d$, then for each $k,n$ the map $j^k_n:G^k_n \rightarrow G^k_\infty$ is surjective.
\item \label{iii} If $p=d$, then for each $k,n$, we have $\ker j^k_n \subseteq \ker j^k_{n,n+1}$.
\end{enumerate}

We wish to show that $F_\infty$ is isomorphic to $\mathbb{Z}$ if $p=d$ and is trivial if $p\neq d$, just as each of the spaces $G^k_\infty$ are. 

Consider first the case $p\neq d$. By (\ref{i}), we have that for all $k$ and for all $m>n$, the maps $j^k_{n,m}$ are trivial. It follows by the diagram that the maps $i_{n,n+3}$ are all trivial (as they factor through $j^k_{n,n+1}$ for some $k$) and therefore that $F_\infty$ is trivial when $p\neq d$. 

Now we consider the case $p=d$. 

\begin{claim}\label{earlysurj}
In degree $p=d$, $i_2:F_2\rightarrow F_\infty$ is surjective.
\end{claim}
\begin{proof}[Proof of Claim \ref{earlysurj}]
This is just diagram-chasing. We will freely use the three properties of the diagram (\ref{diagram}) described above, and we encourage the reader to simply trace the proof in that diagram.

Fix $x\in F_\infty$. Then $x=i_m (x_m)$ for some $m\in\mathbb{N}$, by the definition of the direct limit. Fix $k$ large so that diagram (\ref{diagram}) has diagonal arrows $a_l:=a^k_l$ and $b_l:=b^k_l$ up to $l=m+3$. (We will suppress all superscripts $k$ in the proof of this claim.)

Let $y_{m+1} = a_m (x_m) \in G_{m+1}$. Then some $y_1\in G_1$ satisfies $j_1 (y_1) = j_{m+1}(y_{m+1})$, by (\ref{ii}), and so
$$ j_{m+1}(y_{m+1}) = j_{m+1} j_{1, m+1} (y_1).$$
It follows, by (\ref{iii}), that
$$ j_{m+1, m+2} (y_{m+1}) = j_{1,m+2} (y_1). $$
Denote this element by $y_{m+2} \in G_{m+2}$.

Let $x_{m+3} = b_{m+2} (y_{m+2}) \in F_{m+3}$. We have
$$ i_{2,m+3} b_1 (y_1) = b_{m+2} j_{1,m+2} y_1 = b_{m+2} (y_{m+2}) = x_{m+3}.$$

In addition,
\begin{align*}
x_{m+3} &= b_{m+2}(y_{m+2})\\
&= b_{m+2} j_{m+1, m+2} (y_{m+1})\\
&= b_{m+2} j_{m+1, m+2} a_m (x_m)\\
&= i_{m,m+3} (x_m)
\end{align*}

It follows that $i_{m+3} (x_{m+3}) = i_{m}  (x_m) = x$, and so
$$ i_2 b_1 (y_1) = i_{m+3} i_{2,m+3} b_1 (y_1) = i_{m+3} (x_{m+3}) = x $$

and so $i_2$ is surjective.
\end{proof}

The following claim is also proven by a similar diagram chase.
\begin{claim}\label{earlyinj}
In dimension $p=d$, $\ker i_n \subset \ker i_{n,n+3}$.
\end{claim}
\begin{proof}[Proof of Claim \ref{earlyinj}]
Suppose that for some $x_n\in F_n$, $i_n(x_n)=0$. Then for some $m\geq n$, $i_{n,m}(x_n) = 0$. As in the previous claim, we now fix $k$ large so that diagram (\ref{diagram}) has diagonal arrows $a_l:=a^k_l$ and $b_l:=b^k_l$ up to column $l=m$. We then see that
$$ j_{n+1, m+1} a_n (x_n) = a_m i_{n,m} (x_n) = a_m(0) = 0. $$
By (\ref{ii}) above, it follows that $j_{n+1,n+2}a_n(x_n)=0$. Thus,
$$ i_{n,n+3}(x_n) = b_{n+2} j_{n+1,n+2} a_n (x_n) = b_{n+2}(0) = 0.$$
This completes the proof of Claim \ref{earlyinj}.
\end{proof}

Now fix $k$ so that the diagonal arrows in diagram (\ref{diagram}) exist up to $n=10$. Let $G_\infty = G^k_\infty \cong \mathbb{Z}$. (Now that $k$ is fixed we will again suppress the superscripts $k$.) We now define homomorphisms $\psi_n: F_n \rightarrow G_\infty\cong \mathbb{Z}$ by
$$ \psi_n (x_n) = j_3 a_2 i_2^{-1} i_n (x_n). $$

Note that $i_2$ is surjective but not necessarily injective; nonetheless we have the following fact:

\begin{claim}\label{psiwelldef}
The maps $\psi_n$ are well-defined homomorphisms (i.e. independent of the choice of $i_2^{-1}i_n(x_n)$) and are compatible, in the sense that $\psi_m i_{n,m}(x_n) = \psi_n(x_n)$. 
\end{claim}
\begin{proof}[Proof of Claim \ref{psiwelldef}]
Suppose first that $i_2(x_2) = i_2(x'_2)$ for some $x_2, x'_2\in F_2$. To show that $\psi_n$ is well-defined we must show that
$$ j_3 a_2 (x_2) = j_3 a_2 (x'_2). $$
By Claim \ref{earlyinj}, $i_{2,3} x_2 = i_{2,3} x'_2$. Thus,
$$ j_3 a_2 (x_2) = j_4 a_3 i_{2,3} (x_2) = j_4 a_3 i_{2,3} (x'_2) = j_3 a_2 (x'_2).$$
This shows that $\psi_n$ is well-defined.

To see that $\psi_m(i_{n,m}x_n) = \psi_n(x_n)$, we note that if $i_2(x_2) = i_n(x_n)$, then $i_2(x_2) = i_m i_{n,m} (x_m)$. Thus,
$$\psi_m(i_{n,m}(x_n)) = j_3 a_2 (x_2) = \psi_n(x_n). $$
\end{proof}

It follows that the maps $\psi_n$ induce a homomorphism $h:F_\infty \rightarrow G_\infty\cong\mathbb{Z}$ satisfying $h\circ i_n = \psi_n$ for all $n$. We will show that $h$ is injective and surjective.

Suppose $h(x) = 0$ for $x\in F_\infty$. By Claim \ref{earlysurj}, we can write $x=i_2 x_2$, for $x_2\in F_2$. Therefore,
$$ 0 = h i_2 (x_2) = \psi_2 (x_2) = j_3 a_2 (x_2). $$
Because $\ker j_3 \subseteq \ker j_{3,4}$, we have 
$$ j_{3,4} a_2 (x_2) = 0 \in G_4. $$
It follows that 
$$ i_{2,5} (x_2) = b_4  j_{3,4} a_2 (x_2) = 0 \in F_5$$
and therefore that $x = i_5 i_{2,5} (x_2) = 0\in F_\infty$. This shows that $h$ is injective.

To show that $h$ is surjective, it suffices to show that $\psi_2 = j_3 a_2$ is surjective. Consider $y\in G_\infty$. Because $j_1$ is surjective, $y=j_1 (y_1)$ for some $y_1 \in G_1$. Letting $x_2 = b_1 (y_1)$, we see that $\psi_2 (x_2) = y$. This shows that $h$ is surjective.
\end{proof}

\subsection{Some basic degree theory}\label{localdegree}
In this section, we give a degree-type lemma for close mapping packages. The idea here is quite simple, though the notation is cumbersome: If the limit of a suitable sequence of mappings is a homeomorphism, then sufficiently close limiting mappings should have images which contain a ball of fixed radius.

We now fix a set-up and some notation.

Let $\{(X_k, p_k)\}$ and $\{(Y_k, q_k)\}$ be two sequences of pointed metric spaces converging to $(X,p)$ and $(Y,q)$, respectively. Suppose that all the spaces are uniformly Ahlfors regular, $(L,r_0)$-LLC, homology $d$-manifolds, and furthermore that $\{Y_k\}$ and $Y$ are topological $d$-manifolds. 

For some fixed $R>0$, let $F_k = \overline{B}(p_k, R)$ and assume also that the sequence of pointed metric spaces $\{(F_k,p_k)\}$ converge to the pointed metric space $(F,p)$, where $F\subset X$ and $F\supset B(p,R)$ in $X$. 

Finally, assume that $w_k:F_k \rightarrow Y_k$ are uniformly $C$-Lipschitz and that we have convergence of the sequence of mapping packages:
$$ \left\{ (F_k, p_k), (Y_k, q_k), w_k \right\} \rightarrow \left\{ (F, p), (Y, q), w \right\} $$

By Lemma \ref{ctsalmostiso}, there are continuous mappings
\begin{align*}
f_k&: F \rightarrow X_k,\\
g_k&: F_k \rightarrow X,\\
\tilde{f_k}&: \overline{B}_{Y}(q, 3CR) \rightarrow Y_k,\\
\tilde{g_k}&: \overline{B}_{Y_k}(q_k, 3CR) \rightarrow Y,
\end{align*}
that almost preserve distances and are almost inverses, up to additive error that decreases to zero with $k$.

Fix an open set $A\subseteq F$.

\begin{lemma}\label{degree} 
Suppose that, for some $r, r'<r_0$, the ball $B_X(z,4r)\subset A$, the map $w|_{A}$ is a homeomorphism, and $w(B_X(z,r)) \supseteq B_Y(w(z), r')$. Then for all $k$ sufficiently large, $w_k(B_{X_k}(f_k(z),2r)) \supseteq B_{Y_k}(w_k(f_k(z)), r'/2)$.
\end{lemma}
\begin{proof}

Let $0<\eta<r'/(100L)$. For all $k$ sufficiently large, the maps $f_k$, $g_k$, $\tilde{f}_k$, and $\tilde{g}_k$ preserve distances up to additive error $\eta$. In addition, again by Lemma \ref{ctsalmostiso}, we may assume that, for all $k$ large and for all $r<2CR$, the map
$$\tilde{g}_k \circ \tilde{f}_k|_{B(p,r)}$$
is $\eta$-homotopic to the inclusion map of $B(p,r)$ into $B(p,R)$.

Fix $k\in\mathbb{N}$ large enough for this to hold; from now on, this $k$ will be fixed, so we drop the subscript and denote these maps by $f$, $g$, $\tilde{f}$, and $\tilde{g}$. By Lemma \ref{ctsalmostiso2} and Lemma \cite{Se96}, Proposition 5.8, we can also arrange that the maps  $\tilde{f} \circ w$ and $w_k \circ f$, when restricted to $A$, are $\eta$-homotopic on $F$.

Let $B=\overline{B}_X(z,r)$. Fix $y\in B_{Y_k}(w_k(f_k(z)), r'/2)$.

First, because $w$ is a homeomorphism on $A$, the induced mapping on relative homology,
$$ w_*: H_d(A, A\setminus B) \rightarrow H_d(Y, Y\setminus w(B)) $$
is an isomorphism. Note that $H_d(Y, Y\setminus w(B))$ is non-trivial, by duality (e.g. \cite{Sp81}, Theorem 6.2.17).

Let $V=B_Y(q,2CR)$. The map $\tilde{f}$ induces a non-trivial map
$$ \tilde{f}_* : H_d(V, V \setminus w(B)) \rightarrow H_d(Y_k, Y_k \setminus B'') $$
where $B'' = \overline{B}(y, r'')$, $r''=r'/(10L)$. Indeed, if this map were trivial, then the map
$$ \tilde{g}_* \tilde{f}_* : H_d(V, V \setminus w(B)) \rightarrow H_d(Y, Y \setminus \{p\})$$
would be trivial for some $p\in w(B)$. But this map on homology is the same as that induced by inclusion, so this cannot be the case by the duality argument of Lemma \ref{poinc} (ii).

It follows that the map
$$ (\tilde{f}\circ w)_* = \tilde{f}_* w_* :  H_d(A, A\setminus B) \rightarrow  H_d(Y_k, Y_k \setminus B'') $$
is non-trivial.

Because $\tilde{f}\circ w$ and $w_k \circ f$  are $\eta$-homotopic, the map 
$$(w_k\circ f)_*:  H_d(A, A\setminus B) \rightarrow  H_d(Y_k, Y_k \setminus B'')  $$
is non-trivial.

This implies that 
$$(w_k\circ f)_*:  H_d(A, A\setminus B) \rightarrow  H_d(Y_k, Y_k \setminus \{y\})  $$
is non-trivial. Indeed, if not, then by Lemma \ref{poinc} (iii),
$$(w_k\circ f)_*:  H_d(A, A\setminus B) \rightarrow  H_d(Y_k, Y_k \setminus B'')  $$
would be trivial, but we just showed that it is not.

So we have shown that 
$$(w_k\circ f)_*:  H_d(A, A\setminus B) \rightarrow  H_d(Y_k, Y_k \setminus \{y\})  $$
is non-trivial. It follows from this that $y\in (w_k\circ f)(B)$.

Because $f(B)\subset B(f_k(z), 2r)$, we get that
$$y\in w_k(f(B)) \subset w_k(B(f_k(z), 2r)).$$
\end{proof}

Later on, it will be convenient to work with a cohomological notion of local degree, which we introduce now. The following material is taken from \cite{HR02}. For proofs, see \cite{RR55}.

Let $H^*_c$ denote the Alexander-Spanier cohomology groups with compact supports and coefficients in $\mathbb{Z}$. (For the definition and properties of Alexander-Spanier cohomology, see \cite{Ma78}.) The following definition is taken from \cite{HR02}, I.1.

\begin{definition}\label{genmfld}
A locally compact, Hausdorff, connected, and locally connected space $M$ is called a \textit{generalized} $d$\textit{-manifold} if
\begin{itemize}
\item $H^p_c(U) = 0$ whenever $U\subseteq M$ is open and $p\geq d+1$.
\item For every $x\in M$ and every open neighborhood $U$ of $x$, there is another open neighborhood $V$ of $x$ contained in $U$ such that
$$
H^p_c(V) =
\begin{cases}
\mathbb{Z} & \text{if }p=d \\
0 & \text{if }p=d-1
\end{cases}
$$ 
and the standard homomorphism $H^n_c(W)\rightarrow H^n_c(V)$ is surjective whenever $W$ is an open neighborhood of $x$ contained in $V$.
\item $X$ has finite topological dimension.
\end{itemize}

A generalized $d$-manifold $X$ is called \textit{orientable} if $H^d_c(X) = \mathbb{Z}$, and \textit{oriented} if we fix a choice of generator in $H^d_c(X)$.
\end{definition}

\begin{rmk}\label{homisgen}
Any homology $d$-manifold is a generalized $d$-manifold, as noted in \cite{HR02}, Example 1.4 (c).
\end{rmk}

A generalized $d$-manifold $X$ is said to be \textit{oriented} if $H^d_c(X) = \mathbb{Z}$. In this case we can simultaneously orient all connected open subsets $U$ of $X$ via the isomorphism between $H^d_c(U)$ and $H^d_c(X)$. 

We will not use any sophisticated facts about cohomology below, but only the following object and its basic properties: Let $X$ and $Y$ be oriented generalized $d$-manifolds, and let $f:X\rightarrow Y$ be continuous. For any relatively compact domain $D$ in $X$, and for every $y\in Y\setminus f(\partial D)$, we can associate an integer called the local degree $\mu(y,D,f)$. In the following lemma, we collect the only properties of $\mu$ we will need.

\begin{lemma}\label{ld_props}
For continuous maps $f,g$ between oriented generalized $d$-manifolds $X$ and $Y$, and a relatively compact domain $D\subseteq X$, the local degree $\mu$ has the following properties:
\begin{itemize}
\item The function $y\rightarrow \mu(y,D,f)$ is constant on each connected component of $Y\setminus f(\partial D)$. 

\item If $y\notin f(\overline{D})$, then $\mu(y,D,f)=0$.

\item If $f : D \rightarrow f(D)$ is a homeomorphism, then $\mu(y, D, f ) = \pm 1$ for each $y \in f (D)$.

\item If $y\in Y\setminus f(\partial D)$ and if $f^{-1}(y)\subset D'$, where $D'$ is a domain contained in $D$ such that $y\in Y\setminus f(\partial D')$, then
$$ \mu(y,D,f) = \mu(y,D',f) $$
\end{itemize}
\end{lemma}
\begin{proof}
These facts can all be found in \cite{HR02}, 2.3 or \cite{RR55}, II.2.
\end{proof}

\subsection{The Bonk-Kleiner theorem on mappings of bounded multiplicity}
This material is taken from \cite{BK02_rigidity}.

\begin{definition}\label{defbm}
A map $f$ between spaces $X$ and $Y$ is of \textit{bounded multiplicity} if there is a constant $N\in\mathbb{N}$ such that $\# f^{-1}(y) \leq N$ for all $y\in Y$.
\end{definition}

The following result of Bonk and Kleiner provides a partial substitute, in our setting, for Reshetnyak's theorem on quasi-regular mappings. (See the discussion of David's proof in Subsection \ref{background}.)

\begin{thm}[\cite{BK02_rigidity}, Theorem 3.4] \label{bm}
Suppose $X$ is a compact metric space, every non-empty open subset of $X$ has topological dimension at least $d$, and $f:X \rightarrow \mathbb{R}^d$ is a continuous map of bounded multiplicity. Then there is an open subset $V\subseteq f(X)$ with $\overline{V}=f(X)$ such that $U=f^{-1}(V)$ is dense in $X$ and $f|_U:U\rightarrow V$ is a covering map.
\end{thm}

\section{Warm-up: Getting bi-Lipschitz tangents}\label{warmup}

In this section, we prove a result that is much weaker than Theorem \ref{blp}, but whose proof illustrates some of the techniques used in the proof of Theorem \ref{blp}. Nothing in this section is needed in the proof of Theorem \ref{blp}, so a reader who is solely interested in that proof can skip this section without missing anything needed later in the paper.

Let $((X,d,x), (Y,\rho,y), f)$ denote a mapping package, in the sense of Definition \ref{mappingpackage}. Thus, $(X,d)$ and $(Y,\rho)$ are metric spaces with $x\in X$ and $y\in Y$, and $f:X\rightarrow Y$ is a mapping such that $f(x)=y$. We will also assume that $f$ is Lipschitz.

Define a \textit{tangent of $f$ at $x$} to be a mapping package
$$ ((X_\infty,d_\infty, x_\infty), (Y_\infty, \rho_\infty, y_\infty), f_\infty)$$
for which there is a sequence of real numbers $\lambda_n\rightarrow\infty$ such that, in the sense of Definition \ref{mappingpackageconvergence}, we have
$$ ((X,\lambda_n d, x), (Y, \lambda_n \rho, y), f) \rightarrow ((X_\infty,d_\infty, x_\infty), (Y_\infty, \rho_\infty, y_\infty), f_\infty) $$
as $n\rightarrow\infty$.

Taking a tangent of $f$ at $x$ amounts to blowing up $f$ (as well as its source and target spaces) at the point $x$, along some sequence of scales which grows to infinity. Note that the spaces $X_\infty$ and $Y_\infty$ are also ``weak tangents'' of the spaces $X$ and $Y$, as in Definition \ref{manifoldtangents}.

We will say that $f$ has a bi-Lipschitz tangent at $x$ if, for one of its tangent mapping packages at $x$, the mapping $f_\infty$ which arises is bi-Lipschitz.

Suppose that $f$ is Lipschitz and that $X$ and $Y$ are doubling metric spaces. Suppose also that $X$ is equipped with a doubling measure, and that $x$ is a point of density of a set $E\subset X$ such that $f|_E$ is bi-Lipschitz. Then every tangent of $f$ at $x$ yields a mapping $f_\infty$ that is bi-Lipschitz. This is a standard fact, and its proof is very similar to that given in Proposition \ref{spacefilling} below.

Thus, a mapping having a positive-measure set on which it is bi-Lipschitz is a much stronger condition than a map merely having a bi-Lipschitz tangent. 

In the setting of Theorem \ref{blp}, one can give a simpler argument which shows that the mapping has a bi-Lipschitz tangent. This argument is really contained in \cite{BK02_rigidity}, though our context is slightly different.

In the proof, we will need one definition that we have not yet introduced, coming from Chapter 12 of \cite{DS97}. (This will not be used in the proof of the main Theorem \ref{blp}.)

\begin{definition}\label{DSregular}
A Lipschitz mapping $f:M\rightarrow N$ between two metric spaces is said to be \textit{David-Semmes regular} if there is a constant $C>0$ such that, for every ball $B\subseteq N$ of radius $r$, the set $f^{-1}(B)$ can be covered by at most $C$ balls of radius $Cr$.  
\end{definition}

In particular, David-Semmes regular maps always have bounded multiplicity.

\begin{prop}\label{bltangent}
Let $X$ and $Y$ be Ahlfors $s$-regular, linearly locally contractible, complete, oriented, topological $d$-manifolds, for $s,d\geq 1$. Suppose in addition that $Y$ has $d$-manifold weak tangents.

Suppose that $f:X\rightarrow Y$ has $|f(X)|>0$. Then, for some $x\in X$, $f$ has a bi-Lipschitz tangent.
\end{prop}
\begin{proof}

The first step is to examine the spaces $X_\infty$ and $Y_\infty$. By Proposition \ref{manifoldlimit} and the assumption that $Y$ has $d$-manifold weak tangents, we see that $X_\infty$ is a homology $d$-manifold and $Y_\infty$ is an topological $d$-manifold. We also have, by Proposition \ref{limitLLC}, that $Y_\infty$ is $(L,r_0)$-LLC, for some constants $L$ and $r_0$.

The next step is to apply Proposition 12.8 of \cite{DS97}. This says that, for some $x\in X$, we can find a tangent
$$ ((X_\infty,d_\infty, x_\infty), (Y_\infty, \rho_\infty, y_\infty), f_\infty) $$
of $f$ at $x$ such that $f_\infty$ is a David-Semmes regular map. In particular, this means that $f_\infty$ is a mapping of bounded multiplicity, in the sense of Definition \ref{defbm}.

We would now like to apply Theorem \ref{bm} to $f_\infty$. Fix a small open ball $B\subset X_\infty$. We can choose $B$ so small that $f_\infty(\overline{B})$ lies in a set $V\subset Y_\infty$ which is homeomorphic to an open set in $\mathbb{R}^d$, and which has diameter less than the contractibility radius $r_0$ of $Y_\infty$. 

Let $K=\overline{B}$. Then every open subset of $K$ contains an open subset of the homology $d$-manifold $X_\infty$ and thus has topological dimension at least $d$. Because we also know that $f_\infty$ has bounded multiplicity on $K$, we can apply Theorem \ref{bm}.

In particular, we obtain an open set $U\subset K \subset X$ such that $f_\infty$, when restricted to $U$, is a homeomorphism. Let $V' \subset f_\infty(U)$ be a small open set such that
$$ \dist(V', Y_\infty \setminus f_\infty(U)) > L\diam V' $$
and let $U' = f_\infty^{-1}(V') \cap U$.

We claim that $f_\infty$ is in fact bi-Lipschitz on $U'$. We already know it to be Lipschitz, so it suffices to establish the other bound. Fix $x,y\in U'$ and consider $f(x), f(y)\in V'\subset Y_\infty$. Let $r= \rho_\infty(f(x),f(y))$; note that $r<r_0$ by our assumptions.

First of all, there is a compact connected set $S\subset f(U)$ containing $f(x)$ and $f(y)$ such that $\diam S \leq Lr$. Indeed, by our assumptions, the compact set $\overline{B}(f(x), r)$ is contractible within $\overline{B}(f(x), Lr)$. If $H$ is the homotopy realizing this contractibility, then
$$ S = H( \overline{B}(f(x), r) \times [0,1] ) $$
contains $f(x)$ and $f(y)$ and is compact, connected, and contained in $\overline{B}(f(x), Lr) \subset f_\infty(U).$

Now consider $E=f_\infty^{-1}(S) \cap U$. Because $f_\infty$ is a homeomorphism on $U$, we have that $E$ is a compact, connected set in $U$ that contains $x$ and $y$. Because $f_\infty$ is David-Semmes regular, $E$ is contained in the union of $C$ balls of radius $CLr$ in $X_\infty$. It follows that $\diam E \leq 2C^2L r$.

Thus,
$$ d_\infty(x,y) \leq \diam E \leq 2C^2L r = 2C^2L \rho_\infty(f(x),f(y)), $$
and so $f_\infty$ is bi-Lipschitz on $U'$.

To complete the proof of the Proposition, we take another tangent of $f_\infty$ at a point of $U$. This yields a tangent of $f_\infty$ which is globally bi-Lipschitz. That this is also a tangent of $f$ itself is a standard fact.
\end{proof}


\section{Setting up the proof of Theorem \ref{blp}}
We first introduce the following notation:
$$ \tilde{B}_n(x,r) = \bigcup\{Q\in\Delta_n: Q\cap B(x,r)\neq\emptyset\}$$

By Theorem \ref{davidtheorem}, to prove Theorem \ref{blp} it suffices to show the following proposition, which is just a restatement of David's condition, formulated in Definition \ref{davidscondition}.

\begin{prop}\label{david9}
Let $d\in\mathbb{N}$ and $s>0$. Suppose $(Y,\rho)$ is LLC, Ahlfors $s$-regular, and has $d$-manifold weak tangents. For all $C_0$, $L$, $r_0$, $M$ and for all $\lambda, \gamma\geq 0$, there exist $\Lambda, \eta >0$ such that the following holds:

Let $X$ be a complete, oriented, topological $d$-manifold which is Ahlfors $s$-regular with constant $C_0$ and $(L,r_0)$-LLC. Let $I_0$ be a $0$-cube and $z:I_0\rightarrow Y$ an $M$-Lipschitz map. If $x\in X$, $n\in\mathbb{Z}$, and  $T=\tilde{B}_n(x,\Lambda 2^{n})\subseteq I_0$ satisfies $|z(T)|/|T| \geq \gamma$, then one of the following holds:

\begin{enumerate}[\normalfont (i)]
\item $z(T) \supseteq B(z(x), \lambda 2^{n})$, or
\item there is an $n$-cube $R\subset T$ such that
$$ |z(R)|/|R| \geq (1+2\eta)|z(T)|/|T|$$
\end{enumerate}
\end{prop}

We emphasize that in Proposition \ref{david9} the constants $\Lambda$ and $\eta$ depend only on the ``input'' constants $\lambda$ and $\gamma$, as well as the ``data'' $d, s, C_0, L, r_0, M$, and the space $Y$. 

We will actually prove the following similar statement, which implies Proposition \ref{david9}. (This is analogous to Lemma 4 of \cite{Da88}.)

For $r>0$, define $n_r$ to be the largest integer $n$ such that
\begin{equation}\label{nr}
10 C_0 2^n \leq r.
\end{equation} 

\begin{prop}\label{reduction}
Let $d\in\mathbb{N}$ and $s>0$. Suppose $(Y,\rho)$ is LLC, Ahlfors $s$-regular, and has $d$-manifold weak tangents. For all $C_0, L, r_0$ and for all $\gamma> 0$, there exist $\tau, \sigma >0$ such that the following holds:

Let $X$ be a complete, oriented, topological $d$-manifold which is Ahlfors $s$-regular with constant $C_0$ and $(L,r_0)$-LLC. If $v\in X$, $0<r\leq C_0$, $T=\tilde{B}_{n_r}(v,r)$, and $z:T\rightarrow Y$ is $1$-Lipschitz satisfying $|z(T)|/|T| \geq \gamma$, then one of the following holds:
\begin{enumerate}[\normalfont (i)]
\item $z(T) \supseteq B(z(v), \tau r)$, or
\item there is a dyadic cube $R\subset T$ of diameter at least $\tau r$ such that
$$ |z(R)|/|R| \geq (1+\sigma)|z(T)|/|T|$$
\end{enumerate}
\end{prop}

As before, the constants $\tau$ and $\sigma$ in Proposition \ref{reduction} depend only on $d, s, C_0, L, r_0$, and  $\gamma$, as well as the space $Y$.

\begin{lemma}
Proposition \ref{reduction} implies Proposition \ref{david9}.
\end{lemma}
\begin{proof}

Suppose that Proposition \ref{reduction} is true but that Proposition \ref{david9} fails. The failure of Proposition \ref{david9}, first of all, implies the existence of dimensions $d\in\mathbb{N}$, $s>0$ and a space $(Y,\rho)$. It also implies that for some data $C_0, L, r_0, M$, some constants $\lambda, \gamma>0$ and every $\Lambda, \eta>0$, there exists an Ahlfors $s$-regular, LLC, complete oriented topological $d$-manifold $X$ (with data given by $C_0, L, r_0$), a $0$-cube $I_0\subset X$, and
$$T = \tilde{B}_n(x, \Lambda 2^n) \subset I_0, $$
as well as an $M$-Lipschitz $z:T\rightarrow Y$ with $|z(T)|/|T|\geq \gamma$ such that 
\begin{itemize}
\item $z(T) \not\supset B(z(x), \lambda 2^n)$, and
\item for every $n$-cube $R\subset T$,
$$ |z(R)|/|R| \leq (1+2\eta) |z(T)|/|T|.$$
\end{itemize}

We now reduce to the $1$-Lipschitz case by letting $\tilde{z}:T\rightarrow (Y,\frac{1}{M}\rho)$. Then $\tilde{z}:T\rightarrow (Y,\frac{1}{M}\rho)$ satisfies
\begin{itemize}
\item $|\tilde{z}(T)|/|T|\geq \tilde{\gamma} = \gamma/M^s $
\item $\tilde{z}(T) \not\supset B(z(x), \lambda 2^n/M)$, and
\item for every $n$-cube $R\subset T$,
$$ |\tilde{z}(R)|/|R| \leq (1+2\eta) |\tilde{z}(T)|/|T|.$$
\end{itemize}

Let $T' = \tilde{B}_{n_r}(x,r)$ for $r=\Lambda 2^{n}/10$.  Note that, as $T\subset I_0$, we have $\diam{T} \leq \diam{I_0}$ and so $r\leq C_0$. Also note that $T'\subseteq T$, by a simple triangle inequality argument.

If $\Lambda>20C_0$, then $T'$, and therefore also $T\setminus T'$, is a disjoint union of $n$-cubes. Indeed, in this case $n_r\geq n$, and $T$ is a disjoint union of $n_r$-cubes, each of which is a disjoint union of $n$-cubes. 

It follows from the second property of $\tilde{z}$ above that
$$ \frac{|\tilde{z}(T')|}{|T'|} \leq (1+2\eta)  \frac{|\tilde{z}(T)|}{|T|} $$
and
$$ \frac{|\tilde{z}(T\setminus T')|}{|T\setminus T'|} \leq (1+2\eta) \frac{|\tilde{z}(T)|}{|T|}. $$
Therefore
\begin{align*}
|\tilde{z}(T')| &\geq |\tilde{z}(T)| - |\tilde{z}(T\setminus T')|\\
&\geq |\tilde{z}(T)| - (1+2\eta)\frac{|\tilde{z}(T)|}{|T|}|T\setminus T'|\\
&= \left(|T| - (1+2\eta)|T\setminus T'|\right)\frac{|\tilde{z}(T)|}{|T|}\\
&= \left((1+2\eta)|T'| - 2\eta|T|\right)\frac{|\tilde{z}(T)|}{|T|}\\
&\geq \left((1+2\eta)|T'| - 2\eta C|T'|\right)\frac{|\tilde{z}(T)|}{|T|}\\
&\geq (1-C'\eta)|T'|\frac{|\tilde{z}(T)|}{|T|}\\
&\geq \frac{\tilde{\gamma}}{3} |T'|
\end{align*}
if $\eta$ is small depending on $\gamma$. (Here $C$ and $C'$ depend only on the Ahlfors regularity constants $s$ and $C_0$.)

Now, apply Proposition \ref{reduction} to $\tilde{z}:T'\rightarrow (Y,\frac{1}{M}\rho)$ with $\gamma$ as $\tilde{\gamma}/3$. We obtain $\tau$ and $\sigma$. Note that $\tau$ and $\sigma$ depend only on the data $d,s,C_0,L,r_0, M$, the space $Y$, and the constant $\gamma$.

If $\Lambda>\max\{\frac{10 \lambda}{M\tau}, \frac{10C_0}{\tau}\}$ and $\eta$ is sufficiently small relative to $\sigma$, we get that either
\begin{itemize}
\item $\tilde{z}(T)\supseteq \tilde{z}(T') \supseteq B(\tilde{z}(x), \tau \Lambda 2^n/10) \supset B(\tilde{z}(x), \lambda 2^n/M)$, or
\item there is a dyadic cube $R\subset T'$ of diameter at least $\tau \Lambda 2^{n}$ such that
$$ |\tilde{z}(R)|/|R| \geq (1+\sigma)|\tilde{z}(T')|/|T'| \geq  (1+\sigma)(1-C'\eta)|\tilde{z}(T)|/|T| \geq (1+3\eta)|\tilde{z}(T)|/|T|$$ 
\end{itemize}
In the first case we contradict the assumption that the first conclusion in Proposition \ref{david9} fails. In the second case, note that $R$ is a cube at scale larger than $n$ (because $\tau \Lambda 2^n > C_0 2^n$) and therefore a disjoint union of $n$-cubes. At least one of those $n$-cubes $R'$ must then also satisfy
$$ |\tilde{z}(R')|/|R'| \geq (1+3\eta)|\tilde{z}(T)|/|T|,$$
which contradicts the assumption that the second conclusion of Proposition \ref{david9} fails.
\end{proof}

\section{Proof of Proposition \ref{reduction}}

We will use the notation of the previous section; recall especially the definition of $n_r$ from (\ref{nr}).

Suppose now that Proposition \ref{reduction} is false. Then there exists constants $d,s,C_0, L, r_0, \gamma$, and a space $(Y, \rho)$ that is LLC, Ahlfors $s$-regular and has $d$-manifold weak tangents, such that the following holds:

For every $k\in \mathbb{N}$, there are spaces $Z_k$ that are Ahlfors $s$-regular with constant $C_0$ and that are $(L,r_0)$-LLC, complete oriented topological $d$-manifolds. In addition, there are radii $0<r_k\leq C_0$, subsets $T_k = \tilde{B}_{n_{r_k}}(v_k, r_k)\subset Z_k$ and $1$-Lipschitz maps $z_k:Z_k\rightarrow Y$ satisfying $|z_k(T_k)|_Y \geq \gamma |T_k|_{Z_k}$ and such that:
\begin{enumerate}[(i)]
\item $z_k(T_k)\not\supseteq B(z_k(v_k), \frac{1}{k}r_k)$, and
\item for every dyadic cube $R\subseteq T_k$ of diameter at least $r_k/k$, we have
$$ \frac{|z_k(R)|}{|R|} \leq \left(1+\frac{1}{k}\right)\frac{|z_k(T_k)|}{|T_k|}.$$
\end{enumerate} 

Let $X_k$ be the metric space
$$ \left(Z_k, \frac{1}{r_k} d_{Z_k} \right)$$

Let $S_k\subset X_k$ denote the corresponding rescaled version of $T_k$. Then
$$ B(v_k,1)\subseteq S_k \subseteq B(v_k, 2). $$
and
$$ C_0 \leq |S_k|_{X_k} \leq 2^s C_0 $$

Note that $S_k$ has a dyadic cube decomposition given by the rescaled versions of cubes in $T_k$. The following additional technical fact about this decomposition of $S_k$ is obvious but useful.

\begin{lemma}\label{bigcubes}
For every $0<r<1/20$ and every $k\in\mathbb{N}$, the set $S_k$ can be written as a disjoint union of measurable sets $R_j$ satisfying
\begin{itemize}
\item $(2C_0^2)^{-1} r \leq \diam R_j \leq r$, and
\item $(2C_0^2)^{-s} r^s \leq |R_j|_{X_k} \leq r^s$
\end{itemize}
\end{lemma}
\begin{proof}
Choose $n$ such that
$$ C_0 2^n \leq rr_k \leq 2C_0 2^n $$
If $r<1/20$, then 
$$ 2C_0 2^n \leq 2rr_k < r_k/10 \leq  2C_0 2^{n_{r_k}}$$
and so $n\leq n_{r_k}$.

Therefore, we can write $T_k$ as a disjoint union of dyadic cubes in $\Delta_n$. The rescaled versions of these cubes in $S_k$ are now immediately seen to satisfy the required properties.
\end{proof}

For each $k$, we also consider the rescaled target spaces
$$Y_k = \left(Y, \rho_k\right) = \left(Y, \left(\frac{\gamma|T|}{|z_k(T)|}\right)^{1/s} \frac{1}{r_k}\rho\right).$$

Let $w_k:S_k\rightarrow Y_k$ be the map $z_k$ (making the natural identification between points of $Z_k$ and points of $X_k$). Then each $w_k$ is Lipschitz with constant
$$ \left(\frac{\gamma|T|}{|z_k(T)|}\right)^{1/s} \leq 1. $$
In addition, the maps $w_k$ satisfy 
$$|w_k(S_k)| = \gamma |S_k|$$
for all $k$. (The extra rescaling factor $\left(\frac{\gamma|T|}{|z_k(T)|}\right)^{1/s}$ in the target $Y$ is to ensure this last convenient fact.)

Finally, the two important properties of $z_k$ pass to $w_k$ in the following way:
\begin{equation}\label{noballinw}
w_k(S_k)\not\supseteq B(w_k(v_k), \frac{1}{k})
\end{equation}
and for every dyadic cube $R\subseteq S_k$ of diameter at least $1/k$, we have
\begin{equation}\label{wnoexpand}
\frac{|w_k(R)|}{|R|} \leq \left(1+\frac{1}{k}\right)\frac{|w_k(S_k)|}{|S_k|} = \left(1+\frac{1}{k}\right)\gamma
\end{equation}

Let $F_k = \overline{B(v_k, 1/2)} \subset S_k \subset X_k$. We may now consider the following sequence of mapping packages (see Definition \ref{mappingpackage}):
$$ \left\{((F_k, d_{X_k}, v_k) , (Y, \frac{1}{r_k} \rho_k, w_k(v_k)), w_k)\right\}. $$

Note that all the spaces in the above mapping packages are complete and uniformly doubling, and the mappings $w_k$ are uniformly $1$-Lipschitz. By applying Proposition \ref{sublimit}, we obtain a subsequence of this mapping package that converges to a limit $\left\{(F, d, v), (M, d', q)), w)\right\}$. In addition, by Lemma \ref{subsetconvergence} we may assume that along this subsequence we also have the convergence of the sequence of ambient source spaces $(X_k, d_{X_k}, v_k)$ to a space $(X,d,v)$ that contains $F$ as a subset. (We continue to index this sequence by the original parameter $k$.)

The following diagram may be useful for keeping track of this convergence. The dotted arrows represent convergence of spaces in the sense of Definition \ref{spaceconvergence}.

\begin{equation}\label{convergence}
\begin{tikzcd}[column sep=scriptsize]
X_k\arrow[dotted]{d}&\supset& S_k &\supset& F_k \arrow[dotted]{d}\arrow{r}{w_k} &Y_k\arrow[dotted]{d}\\
X &\phantom{\supset} &\phantom{S} &\supset &F \arrow{r}{w} &M
\end{tikzcd}
\end{equation}

We now know, by Proposition \ref{manifoldlimit}, that the space $X$ is an LLC, Ahlfors $s$-regular, homology $d$-manifold. In addition, by Lemmas \ref{ARlimit} and \ref{limitLLC} and the assumption that $Y$  has $d$-manifold weak tangents, the space $M$ is an Ahlfors $s$-regular, LLC, topological $d$-manifold. Finally, it is clear that the set $F\subset X$ contains the open ball $B(v,1/2)$.


The space $X$ is a generalized manifold (see Definition \ref{genmfld}), so we may now fix an open subset of $B(v,1/2)\subset F$ which has $H^d_c$ isomorphic to $\mathbb{Z}$, i.e. is itself an oriented generalized $d$-manifold. We will only work in this oriented subset of $X$ from now on.

Let $A$ be a small open ball in $X$ (of diameter smaller than half the contractibility radius of $X$) centered at $v$ and compactly contained in this oriented open subset. 
Because $M$ is a manifold and $w$ is Lipschitz, by making $A$ small enough, we may assume that $w(\overline{A})$ lies in a single chart of $M$. Let $K=\overline{A}$
, which is compact.

We now investigate the limit map $w$.

\begin{lemma}\label{wboundedmult}
The map $w|_K$ is of bounded multiplicity on $K$. In other words, there exists $N\in\mathbb{N}$ such that for every $x\in\mathbb M$, there are at most $N$ points in $w^{-1}(x) \cap K$.
\end{lemma}
\begin{proof}
We will show that there exists $N$ such that for all $r<1/20$ and every $y\in M$, $w^{-1}(B(y,r)) \cap K$ is contained in the union of $N$ balls of radius $r$ in $X$. This clearly suffices to prove the lemma. (This essentially shows the stronger statement that $w$ is a David-Semmes regular mapping, as in Definition \ref{DSregular}, but we do not need this here.)

Recall from Propositions \ref{almostisos1} and \ref{almostisos} that there are ``almost-isometries'' $\phi_k: F\rightarrow F_k\subset X_k$ and $\sigma_k:Y\rightarrow Y_k$, which, on some fixed ball, preserve distances up to an additive error that tends to zero as $k$ approaches infinity. In addition, it follows immediately from those propositions that
$$ \lim_{k\rightarrow\infty} \rho_k\left(w_k(\phi_k(x)), \sigma_k(w(x))\right) = 0$$
locally uniformly on $F\subset X$. 

Fix a ball $B(y,r)$ in $M$. Let $E = w^{-1}(B(y,r)) \cap K$. Let $E_k = \phi_k(E) \subseteq X_k$. Note that if $k$ is sufficiently large, we have both that $E_k\subset S_k$ and $w_k(E_k) \subset B(\sigma_k(y),2r)$. By Lemma \ref{bigcubes} we may write $S_k$ as a disjoint union of cubes $Q$, each satisfying
$$ (2C_0^2)^{-1} r \leq \diam Q \leq r $$
and
$$ (2C_0^2)^{-d} r^s \leq |Q| \leq r^s. $$
We will call these cubes ``$r$-sized''.

Let $\mathcal{Q}$ denote the collection of $r$-sized cubes in $S_k$ that intersect $E_k$, and let $N_k = \#\{Q\in\mathcal{Q}\}$. Because $w_k$ is $1$-Lipschitz on $S_k$,
$$w_k(Q) \subset  B(\sigma_k(y),(2+2C_0)r)\subset Y_k$$
for all $Q\in\mathcal{Q}$.

Therefore, dividing $S_k$ into those $r$-sized cubes that are in $\mathcal{Q}$ and those that are not (and taking all Hausdorff measures with respect to $X_k$ and $Y_k$) we see that
\begin{align*}
\gamma|S_k| = |w_k(S_k)| &\leq \left|\bigcup_{Q\in\mathcal{Q}} w_k(Q)\right| + \left|\bigcup_{Q\notin \mathcal{Q}}w_k(Q)\right|\\
&\leq |B(\sigma_k(y),(2+2C_0)r)| + \gamma(1+1/k)\sum_{Q\subset S_k, Q\in\Delta_{n_r}\setminus\mathcal{Q}} |Q|\\
&= |B(\sigma_k(y),(2+2C_0)r)| + \gamma(1+1/k)\left(|S_k| - \sum_{Q\in\mathcal{Q}} |Q|\right)\\
&\leq  C_1 r^s + \gamma(1+1/k)\left(|S_k| - N_kC_2 r^s\right)
\end{align*}
where $C_1$ depends only on $C_0$ and the Ahlfors-regularity constant of $M$, and $C_2= (2C_0^2)^{-s}$.

Rearranging this inequality yields
$$ N_k \leq \frac{C_1 r^s + \frac{1}{k}|S_k|}{\gamma(1+\frac{1}{k})C_2 r^s}.$$
 
Because the measures $|S_k|$ are uniformly bounded, we see that for all $k$ sufficiently large (depending on $r$, but that is fine), we have
$$N_k\leq \frac{2C_1}{C_2 \gamma}.$$

Since each cube in $\mathcal{Q}$ is contained in a ball of radius $2C_0r$ in $X_k$, and each $X_k$ is doubling with constant depending only on $C_0$ and $d$, we get that $E_k$ is contained in a union of $N$ balls of radius $r$, where $N$ depends only on $s, C_0$ and $\gamma$. (This holds for all $k$ sufficiently large.)

It immediately follows that the same holds for $E$ (with a possibly larger $N$) by using the distance-preserving properties of $\psi_k$ and $\phi_k$ for $k$ large, and the fact that $X$ is doubling. This proves the lemma.
\end{proof}

\begin{rmk}
In the proof of Lemma \ref{wboundedmult}, we used the fact that $w$ is a limit of mappings $w_k$, each of which does not multiply the measure of cubes of size at least $1/k$ by much more than the factor $\gamma$. The proof would be somewhat simpler if we knew that $w$ itself does not expand the measure of any cube by more than a factor $\gamma$, because then the computations above could all be carried out in the limit $w:X\rightarrow M$, rather than in the limiting objects $w_k:X_k\rightarrow Y_k$. Unfortunately, it is not clear that this ``non-expanding'' property of the maps $w_k$ passes directly to the limit map $w$. The same issue arises in Lemma \ref{1preimage} below.
\end{rmk}

Note now that the set $K$ is a compact set that is the closure of an open set in the homology $d$-manifold $X$. It follows that every relatively open subset of $K$ contains an open subset of $X$ and thus has topological dimension at least $d$ (see \cite{HR02}, Remark 1.3(b)). Recall our assumption that $w(K)$ lies in a single chart of $M$. As $w$ has bounded multiplicity on $K$, we can apply Theorem \ref{bm} to obtain a dense open subset $V$ in $w(K)$ such that $U=w^{-1}(V)\cap K$ is dense in $K$ and $w|_U$ is a covering map.

\begin{lemma}\label{1preimage}
Every point in $V$ has exactly one pre-image in $K$ under $w$.
\end{lemma}
\begin{proof}
In other words, what we must show is that if $x\in U$ and $x'\in K$ with $x'\neq x$, then $w(x)\neq w(x')$. Suppose to the contrary that $w(x)=w(x')=y\in V$. As $x\in U$ and $w$ is a covering map when restricted to $U$, we obtain a ball $B(x,r)\subset U$ such that $w|_{B(x,r)}$ is a homeomorphism and $w(B(x,r))$ contains a ball $B(y,r')\subset M$. Without loss of generality, we may take $r<d(x,x')/10C_0$ and $r<1/20$.

Recall the the continuous ``almost isometries'' $f_k:K\rightarrow X$ from Lemma \ref{ctsalmostiso}. By Lemma \ref{degree}, for all $k$ sufficiently large, we obtain $x_k=f_k(x)\in S_k$ such that $w_k(B(x_k,2r))$ contains the ball $B(y_k,r'/2)\subset Y_k$, where $y_k=w_k(x_k)$. Also let $x'_k = f_k(x')$. For all $k$ large, we have $\rho_k(w_k(x_k), w_k(x'_k))<r/10$, because $w(x) = w(x')$.

Let $r_1=\min\{r,r'\}$. By Lemma \ref{bigcubes}, we may write $S_k$ as the disjoint union of sets $Q$ such that
$$ (2C_0^2)^{-1} r_1/10 \leq \diam Q \leq r_1/10 $$
and
$$ (2C_0^2)^{-d} (r_1/10)^s \leq |Q| \leq (r_1/10)^s. $$
One of these sets $Q$ contains the point $x'_k$; let $Q_0$ denote that set. In addition, let $T$ be the union of all these sets $Q$ that intersect $B(x_k, 2r)$. Note that $Q_0$ is not in $T$ by our choice of $r$ and $r_1$. Then
$$w_k(Q_0)\subset B(y_k,r'/2)\subset w_k(T).$$

We now sum over all the sets $Q$ in $S_k$ as above that are not $Q_0$. Because $w_k(Q_0) \subseteq \bigcup_{Q\neq Q_0} w_k(Q)$, we have that
\begin{align*}
\gamma|S_k| = |w_k(S_k)| &\leq \sum_{Q\neq Q_0} |w_k(Q)|\\ 
&\leq \gamma(1+1/k)\sum_{Q\neq Q_0} |Q|\\
&\leq \gamma(1+1/k)(|S_k| - C_3 r_1^s)
\end{align*}
where $C_3 = (2C_0^2)^{-s}$.

Rearranging and recalling that $|S_k|\leq C_02^s = C_4$, we get
$$ \gamma C_3 r_1^s \leq \frac{\gamma}{k}(C_4 - C_3r_1^s), $$
which is a contradiction for $k$ large.
\end{proof}

\begin{lemma}\label{open}
The map $w|_A:A\rightarrow\mathbb M$ is an open mapping.
\end{lemma}
\begin{proof}
We use the notion of local degree defined in Subsection \ref{localdegree}, which we may apply to the oriented generalized $d$-manifold containing $A$. 

Suppose $w$ is not an open mapping on $A$. Then there is a point $x\in A$ and an open set $G\subseteq A$ containing $x$ such that $y=w(x)$ is not an interior point of $w(G)$. Since $w$ has bounded multiplicity, we can find a closed ball in $G$ containing $x$ and no other pre-images of $y$. Let $B$ be a connected open subset of this ball containing $x$. Then $\overline{B} \cap w^{-1}(y) = \{x\}$.

We now claim that the local degree $\mu(y,B,w)$ is $0$. Suppose to the contrary that $\mu(y,B,w)\neq 0$. Choose a small connected neighborhood $N$ of $y$ that does not intersect the compact set $w(\partial B)$. Then $\mu(y',B,w)\neq 0$ for all $y'\in N$. It follows (by Lemma \ref{ld_props}) that $N\subseteq w(B)$, which contradicts our assumption that $y\notin \text{int}(w(G))$. So $\mu(y,B,w)=0$.

On the other hand, we can choose $x'\in B\cap U$ so that $y'=w(x')\in V$ is arbitrarily close to $y$. By Lemma \ref{1preimage}, $x'$ is the only pre-image of $y'$. As before, choose a small connected neighborhood $B'\subset B$ around $x'$ so that $w|_{B'}$ is a homeomorphism and $\partial B'$ avoids the (finitely many) pre-images of $y$. Remember that $B'$ contains the only pre-image $x'$ of $y'$ in $B$. It follows from Lemma \ref{ld_props} that 
$$ \mu(y', B, w) = \mu(y', B', w) = \pm 1. $$

Now, if $y'$ is sufficiently close to $y$, then $y'$ is in the same connected component of $M\setminus w(\partial B)$ as $y$. Because the local degree is locally constant (Lemma \ref{ld_props}), we see that
$$ \mu(y,B,w) = \mu(y',B,w). $$
But the left-hand side is $0$ while the right-hand side is not. This completes the proof that $w$ is an open mapping.
\end{proof}

From the previous two lemmas it immediately follows that $w$ is a homeomorphism on $A$. Indeed, we only need show it is injective. Suppose $w(x) = w(x')$. Choose small disjoint balls $B$ and $B'$ containing $x$ and $x'$, respectively. Then $w(B)\cap w(B')$ is an open set in $w(A)$ and therefore contains a point of $V$. This contradicts Lemma \ref{1preimage}.

Because $w$ is a homeomorphism, there are radii $r,r'>0$ such that $w(B(v,r)) \supseteq B(w(v),r')$. It follows by Lemma \ref{degree} and Lemma \ref{ctsalmostiso} that for all $k$ sufficiently large,
$$w_k(S_k) \supseteq w_k(B(f_k(v), 2r)) \supseteq B(w_k(f_k(v)),r'/2) \supseteq B(w_k(v_k), r'/3).$$

This contradicts property (\ref{noballinw}) of $w_k$ if $k$ is large enough.

This completes the proof of Proposition \ref{reduction} and thus of Theorem \ref{blp}.

\section{Proof of Theorem \ref{rect}}

Let $X$ be an Ahlfors $d$-regular, LLC, oriented topological $d$-manifold. (We re-emphasize the fact that here the Ahlfors regularity dimension and the topological dimension of $X$ must coincide.) We will apply Theorem \ref{blp} (in the case $Y=\mathbb{R}^d$) to a class of maps on $X$ provided by a theorem of Semmes. These are given in the following result, which is a slightly weakened version of Theorem 1.29(a) of \cite{Se96}.

\begin{thm}[\cite{Se96}, Theorem 1.29(a)]\label{bubble}
Let $B$ be an open ball in $X$ of radius $r>0$. Then there is a surjective Lipschitz map $f$ from $X$ onto the standard $d$-dimensional unit sphere $\mathbb{S}^d$ with Lipschitz constant $\leq Cr^{-1}$ that is constant on $X\setminus B$. The constant $C$ depends only on the data of $X$.
\end{thm}

\begin{rmk}
In Theorem \ref{bubble}, it makes no difference whether one endows $\mathbb{S}^d$ with the standard Riemannian metric of diameter $\pi$ or with the ``chordal'' metric arising from writing $\mathbb{S}^d = \{x\in\mathbb{R}^{d+1}: |x|=1\}$ and letting $d(x,y) = |x-y|$. These metrics are bi-Lipschitz equivalent. For convenience, we will use the latter.
\end{rmk}

\begin{proof}[Proof of Theorem \ref{rect}]
As above, write $\mathbb{S}^d = \{x\in\mathbb{R}^{d+1}: |x|=1\}$. Consider the projection $p$ from $\mathbb{S}^d$ onto the first $d$ coordinates in $\mathbb{R}^{d+1}$. Then $p$ is $1$-Lipschitz and $|p(\mathbb{S}^d)| = \sigma_d$, the $d$-dimensional Hausdorff measure of the unit ball in $\mathbb{R}^d$.

Therefore, by post-composing the maps of Theorem \ref{bubble} with $p$, we see that for every ball $B(x,r)\subseteq X$ there is a $Cr^{-1}$-Lipschitz map $g_B:B\rightarrow \mathbb{R}^d$ with $|g_B(B)| = \sigma_d$.

To show $X$ is locally uniformly rectifiable, we must show that for all $R>0$ there exists constants $\alpha, \beta$ such that for every ball $B$ of radius at most $R$, there is a set $E\subseteq B$ and a map $f:E\rightarrow\mathbb{R}^d$ such that $|E|\geq \beta|B|$ and $f$ is $\alpha$-bi-Lipschitz.

Fix a ball $B=B(x,r)$, where $r<R$. Let $n$ be such that $C_0 2^n < r \leq C_0 2^{n+1}$. Then $B$ contains a dyadic cube $Q\in \Delta_n$. 

As $c_0 2^n \geq \frac{c_0}{2C_0} r$, $Q$ contains a ball $B'$ of radius $\frac{c_0}{2C_0} r$. Let $g=g_{B'}$ be a map as above associated to $B'$. Then $g$ is Lipschitz with Lipschitz constant bounded by $\frac{2CC_0}{c_0r}$.

Therefore, the map $h = \frac{c_0 r}{2CC_0} g$ is $1$-Lipschitz and $|h(B')| \geq c_5 r^d$, for $c_5 = \sigma_d(c_0/2CC_0)^d$.

Thus, $|h(Q)| \geq \delta|Q|$ for some constant $\delta$ depending only on the data of $X$. By choosing $\epsilon>0$ sufficiently small in Theorem \ref{blp} (see Remark \ref{allcubes}) we get that $h$ is $\alpha$-bi-Lipschitz on a set $E\subset Q\subset B$ of measure at least $\theta|Q|\geq \beta|B|$, where $\alpha$ and $\beta$ depend only on $R$ and the data of $X$. This proves Theorem \ref{rect}.
\end{proof}

\section{Consequences of Theorem \ref{rect}}

It is now possible to derive many corollaries which result immediately from applying deep theorems of David and Semmes on uniformly rectifiable sets to the conclusion of Theorem \ref{rect}. We state two geometric examples below.

First of all, Theorem \ref{rect}, in combination with a result of Semmes in \cite{Se99}, provides a quasisymmetric embedding result for suitable compact metric manifolds. For the definition and basic properties of quasisymmetric homeomorphisms, see \cite{He01}.

\begin{cor}\label{embed}
Let $X$ be an Ahlfors $d$-regular, LLC, compact, oriented topological $d$-manifold. Then $X$ is quasisymmetrically equivalent to a space $X'$ that is also an Ahlfors $d$-regular, LLC, compact, oriented topological $d$-manifold and that is a subset of some $\mathbb{R}^N$.
\end{cor}
\begin{proof}
By Theorem \ref{rect}, the space $X$ is uniformly rectifiable. Proposition 2.10 of \cite{Se99}, combined with equation (3.27) in that paper, shows that $X$ can be quasisymmetrically deformed by a weight so that the resulting space admits a bi-Lipschitz embedding into some $\mathbb{R}^N$.

Both the deformation and the bi-Lipschitz embedding preserve the Ahlfors $s$-regularity of $X$. For the former, this is explained in the discussion following the proof of Lemma 4.4 in \cite{Se99}; the latter is a general fact about bi-Lipschitz mappings.

Thus, if we let $X'$ be the image of the deformed $X$ under the bi-Lipschitz embedding, then $X'$ is Ahlfors $s$-regular. Because it is quasisymmetrically homeomorphic to $X$, it is also a compact, LLC, oriented topological $d$-manifold.
\end{proof}

\begin{rmk} 
Every doubling metric space quasisymmetrically embeds in some Euclidean space by Assouad's theorem (see \cite{He01}, Theorem 12.2), but in general this embedding first ``snowflakes'' the metric, increasing the Hausdorff dimension and destroying the rectifiability properties of the space. Corollary \ref{embed} is false if one replaces ``quasisymmetrically'' by ``bi-Lipschitz'', as examples of Semmes \cite{Se96_bilip} and Laakso \cite{La02} show. 
\end{rmk}

Once there is a nice embedding of the abstract metric space $X$ as a uniformly rectifiable subset of Euclidean space, all the theory of these sets developed by David and Semmes can be applied. Here we merely mention one further example, which says that the image of the embedding in Corollary \ref{embed} can be taken to lie in a particularly nice subset of $\mathbb{R}^N$.

Recall the definition of David-Semmes regular maps, introduced in Definition \ref{DSregular}. We define the following class of subsets of Euclidean space.

\begin{definition}
Let $E$ be an Ahlfors $d$-regular subset of $\mathbb{R}^n$. We say that $E$ is \textit{quasisymmetrically $d$-regular} if $E = g(f(\mathbb{R}^d))$, where $f:\mathbb{R}^d\rightarrow Y$ is a quasisymmetric homeomorphism of $\mathbb{R}^d$ onto an Ahlfors $d$-regular space $Y$, and $g:Y\rightarrow\mathbb{R}^N$ is a David-Semmes regular mapping.
\end{definition}

Quasisymmetrically $d$-regular sets admit bounded-multiplicity parametrizations by $\mathbb{R}^d$ in a controlled way.

The following corollary follows from a weakened version of the implication (C6)$\Rightarrow$(C7) in the main result of \cite{DS91}. (The full version of the result should discuss deformations by $A_1$-weights, which we have not mentioned.)

\begin{cor}\label{qsregular}
Let $X$ be an Ahlfors $d$-regular, LLC, compact, oriented topological $d$-manifold. Let $X'$ be a quasisymmetrically equivalent subset of $\mathbb{R}^N$ provided by Corollary \ref{embed}. Assume $N\geq 2d$. Then $X'$ is a contained in a quasisymmetrically $d$-regular set $E\subset \mathbb{R}^N$.
\end{cor}
\begin{proof}
This follows from Corollary \ref{embed}, Theorem \ref{rect}, and the main result of \cite{DS91} (specifically, the implication (C6)$\Rightarrow$(C7)).
\end{proof}

In general, it is not possible to find good (e.g. quasisymmetric or bi-Lipschitz) parametrizations of metric spaces such as those in Corollary \ref{qsregular} by standard spaces such as $\mathbb{S}^d$ or $\mathbb{R}^d$. Corollary \ref{qsregular} provides a weaker form of ``parametrization'', in that it yields a mapping onto but not into the space, and that is bounded-multiplicity rather than injective.

\section{Counterexamples}\label{counterex}
To conclude, we wish to briefly describe some counterexamples regarding the class of ``Lipschitz implies bi-Lipschitz'' theorems mentioned at the beginning of this paper. By this we mean the class of theorems that say that if $f:X\rightarrow Y$ is a Lipschitz mapping with positive-measure image, then $f$ is bi-Lipschitz on a set of positive measure, quantitatively. None of these counterexamples are new, but they are scattered in a few different places in the literature and it may be convenient to collect them in one place.  The first two can be found in Meyerson's paper \cite{Me13}, the third is due to David and Semmes \cite{DS97}, and the fourth is an example of Laakso \cite{La99}.

The first counterexample shows that, in the setting of Theorem \ref{blp}, the requirement that the two spaces have the same topological dimension is necessary. This proposition is proven by Meyerson in \cite{Me13}, Theorem 4.1. Here we give a slightly different argument.

\begin{prop}\label{spacefilling}
There is an Ahlfors $2$-regular, linearly locally contractible, complete oriented topological $1$-manifold $X$ and a Lipschitz map $f:X\rightarrow \mathbb{R}^2$ with positive measure image that is not bi-Lipschitz on any subset of positive measure.
\end{prop}
\begin{proof}
The metric space $X$ will be the ``snowflaked'' space $(\mathbb{R}, |\cdot|^{1/2})$, equipped with two-dimensional Hausdorff measure (which is the same as one-dimensional Hausdorff measure on $(\mathbb{R},|\cdot|)$). It is clear that $X$ satisfies all the required properties.

It is well-known (see, e.g., \cite{SS05}, Theorem 7.3.1) that there is a space-filling curve $f:(\mathbb{R}, |\cdot|)\rightarrow \mathbb{R}^2$ that is H\"older continuous with exponent $1/2$ and whose image contains the unit square in $\mathbb{R}^2$. Therefore, when considered as a mappping $f:X\rightarrow\mathbb{R}^2$, $f$ is Lipschitz, and it has positive-measure image.

However, no Lipschitz map from $X$ to $\mathbb{R}^2$ can be bi-Lipschitz on a set of positive measure. Indeed, suppose that $f$ is bi-Lipschitz on a set of positive measure $E$ in $X$, with $f(0)=0$. Let $E'=f(E)\subseteq \mathbb{R}^2$. Without loss of generality, we may assume that $E$ is compact, that $0\in\mathbb{R}$ is a point of density of $E$ in $X$, and that $f(0) = 0\in\mathbb{R}^2$ is a point of density of $E'$ in $\mathbb{R}^2$. (We can always find such points.)

We now consider the sequences of mapping packages
\begin{equation}\label{counterexpackages}
\left\{\left( (E, \frac{1}{n^2}d_X, 0) , (E', \frac{1}{n}|\cdot|, 0), f )\right) \right\}.
\end{equation}

Because $0\in X$ is a point of density of $E$ and $0\in\mathbb{R}^2$ is a point of density of $E'$, we have by \cite{DS97}, Lemmas 9.12 and 9.13, that, in the sense of pointed metric spaces,
$$  \left(E, \frac{1}{n^2}d_X, 0\right) \rightarrow \left(X, d_X, 0\right) $$
and 
$$ \left(E', \frac{1}{n}|\cdot|, 0\right)\rightarrow \left(\mathbb{R}^2, |\cdot|, 0\right). $$ 

Therefore, some subsequence of the sequence of mapping packages in (\ref{counterexpackages}) converges to a mapping package
$$ \left( (X,d_X, 0) , (\mathbb{R}^2, |\cdot|, 0), g \right). $$

The mapping $g$ is bi-Lipschitz, because $f|_{E}$ is bi-Lipschitz. In addition, the map $g$ is surjective. We may see this by passing to another subsequence along which the sequence of inverse mapping packages
$$ \left\{\left( (E', \frac{1}{n}|\cdot|, 0) , (E, \frac{1}{n^2}d_X, 0), (f|_{E})^{-1} \right) \right\}$$
converges to a mapping package
$$ \left( (\mathbb{R}^2, |\cdot|, 0), (X, d_X, 0), h \right). $$
It is then easy to see that $g(h(y))=y$ for all $y\in\mathbb{R}^2$ and therefore that $g$ is surjective.

So $g$ is a bi-Lipschitz homeomorphism of $X$ onto $\mathbb{R}^2$. But this is impossible, as $X$ is homeomorphic to $\mathbb{R}$.
\end{proof}

The two spaces in Proposition \ref{spacefilling} satisfy all the conditions of Theorem \ref{blp}, except that they are manifolds of different topological dimensions.

For the remaining three counterexamples that we mention here, we merely indicate the statements and refer the reader to the original sources for the proofs.

The second example is Theorem 4.2 of \cite{Me13}. Let us first note that, as a consequence of Theorem \ref{blp}, we know the following: Let $X$ and $Y$ be spaces as in Theorem \ref{blp}. Let $U\subset X$ be an open set, and let $f:U\rightarrow Y$ be Lipschitz and satisfy $|f(U)|>0$. Then there is a countable collection of measurable sets $E_i\subset U$ such that $f|_{E_i}$ is bi-Lipschitz for each $i$ and $|f(U\setminus \cup E_i)|=0$. (Here the sets $E_i$ are not necessarily disjoint.) On the other hand, we have the following fact:

\begin{prop}[\cite{Me13}, Theorem 4.2]\label{grushin}
There is a doubling, LLC, complete, oriented topological $2$-manifold $X$ of Hausdorff dimension $2$, an open set $U\subset X$, and a Lipschitz map $f:U\rightarrow\mathbb{R}^2$ that cannot be represented in the above manner. In other words, there is no countable collection of measurable sets $E_i\subset U$ such that $f|_{E_i}$ is bi-Lipschitz for each $i$ and $|f(U\setminus \cup E_i)|=0$.
\end{prop}
In particular, the conclusion of Theorem \ref{blp} does not hold for this choice of $X$ and $Y=\mathbb{R}^2$. In this result, the space $X$ can be chosen to be the sub-Riemannian manifold known as the Grushin plane. The source and target spaces in Proposition \ref{grushin} satisfy all the conditions of Theorem \ref{blp}, except that the source $X$ is not Ahlfors $2$-regular. The idea behind Proposition \ref{grushin} is to reduce to Proposition \ref{spacefilling}, because the Grushin plane $X$ contains a bi-Lipschitz equivalent copy of the snowflaked line $(\mathbb{R}, |\cdot|^{1/2})$ as a positive-measure subset.

If one completely relaxes the strong topological conditions imposed in Theorem \ref{blp}, then one can find Lipschitz mappings between metric spaces with large images but no bi-Lipschitz pieces, even in the presence of very strong analytic assumptions on the spaces and mappings.

\begin{prop}[\cite{DS97}, Proposition 14.5]\label{DScounterex}
There is a compact, Ahlfors regular metric space $X$ and a Lipschitz mapping $f:X\rightarrow X$ which is not bi-Lipschitz on any positive-measure subset. Furthermore, the mapping $f$ can be taken to be a homeomorphism which is in addition David-Semmes regular and preserves measure, in the sense that $|f(K)|=|K|$ for all compact $K\subseteq X$.
\end{prop}

The space $X$ in Proposition \ref{DScounterex} is a totally disconnected Cantor set. See Chapter 14 of \cite{DS97} for the proof and some other related constructions.

In both the positive result Theorem \ref{blp} and the counterexample Propositions \ref{spacefilling} and \ref{DScounterex}, the spaces in question may have no ``good calculus'', i.e. they may have no rectifiable curves and therefore no Poincar\'e inequality. (For the definition of Poincar\'e inequalities on metric measure spaces, see \cite{He01}.) It is not known to what extent this type of calculus is helpful in proving ``Lipschitz implies bi-Lipschitz'' theorems, but in closing we wish to note the following theorem of Laakso \cite{La99}, which shows that Ahlfors regular spaces with Poincar\'e inequalities may still fail to have such results.

\begin{prop}[\cite{La99}]
There exists an Ahlfors regular space $X$ admitting a Poincar\'e inequality and a Lipschitz map $f:X\rightarrow X$ with positive-measure image such that there is no positive-measure subset of $X$ on which $f$ is bi-Lipschitz.
\end{prop}

In fact, in Laakso's example the mapping $f$ does not even have any bi-Lipschitz tangents, in the sense of Section \ref{warmup}.

\bibliography{GCDbibliography}{}
\bibliographystyle{plain}

\end{document}